\ifx \processToSlides \undefined
  \ifx\usepackage\RequirePackage    % documentclass already done...  Assume
    \let\usingAmsArtXII\usepackage  % this file is processed by amsart
  \fi
\else
  \documentclass[12pt]{amsart}
  \usepackage{amssymb,amscd}
  \usepackage[letterpaper,margin=2cm,includeheadfoot]{geometry}
  \def \useHugeSize {}
  \def \numberingIsThrough {}
\fi

\ifx \labelsONmargin\undefined
  \def\germ#1{{\mathfrak{#1}}}
  
  \ifx\RequirePackage\undefined
    \expandafter\ifx\csname amsart.sty\endcsname\relax
        \documentstyle[12pt,amssymb,amscd]{amsart}
    \fi

    \def\atSign{@@}

    \def\mathbb{\Bbb}
    \def\mathfrak{\frak}
    \def\mathbf{\bold}
    \ifx\boldsymbol\undefined   % I *think* some very old AMS* did not have it
      \def\boldsymbol#1{{\bold #1}}
    \fi

    \def\mathbit{\boldsymbol}

    \makeatletter
    \newenvironment{proof}{%
         \@ifnextchar[{%
                       \expandafter\let\expandafter\end@proof
                         \csname endpf*\endcsname
                         \my@proof
                      }{\let\end@proof\endpf\pf}%
        }{\end@proof}
    \def\my@proof[#1]{\@nameuse{pf*}{#1}}
    \makeatother

    \def\xrightarrow[#1]#2{@>{#2}>{#1}>}
    \def\xleftarrow[#1]#2{@<{#2}<{#1}<}

    \def\providecommand#1{\def#1}
    \def\emph#1{{\em #1}}
    \def\textbf#1{{\bf #1}}

    \def\mathring{\overset{\,\,{}_\circ}}% For slanted letters only, sub too high
  \else
      \ifx\usepackage\RequirePackage
      \else
        \documentclass[12pt]{amsart}
        \usepackage{amssymb,amscd}
    \let\usingAmsArtXII\usepackage
      \fi
      \ifx \mathring\undefined      % Not AmS-LaTeX 2
        \DeclareMathAccent{\mathring}{\mathalpha}{operators}{"17}
      \fi

      \long\def\FAKEendPROOF{\endtrivlist}
      \ifx\endproof\FAKEendPROOF
      \def\endproof{\qed\endtrivlist}
      \fi

      \ifx\mathbit\undefined    % They can collect their senses and define it
        \DeclareMathAlphabet{\mathbit}{OML}{cmm}{b}{it}
      \fi

      \def\atSign{@}

      \ifx \usingAmsArtXII \undefined
      \else             % We loaded this package ourselves
    \advance\oddsidemargin -0.1\textwidth
    \evensidemargin\oddsidemargin
      \fi

      \def\Sb#1\endSb{_{\substack{#1}}}
      \def\Sp#1\endSp{^{\substack{#1}}}

  \fi
\makeatletter
\@ifundefined{DeclareFontShape}
     {\@ifundefined{selectfont}
             {%    Here we are in non-NFSS
                \message{We need AmS-LaTeX, this version of LaTeX does not
                        support it}\@@end
             }{%  Here we are in  NFSS1
                \def\mathcal{\cal}
                \def\cyr{\protect\pcyr}
                \def\pcyr{%
                        \def\default@family{UWCyr}%
                        \let\oldSl@\sl
                        \def\sl{\def\default@shape{it}\oldSl@}%
                        \cyracc
                        \language\Russian\family{UWCyr}\selectfont
                }
             }}{% Here we are in  NFSS2
                \DeclareFontEncoding{OT2}{}{\noaccents@}
                \DeclareFontSubstitution{OT2}{cmr}{m}{n}
                \DeclareFontFamily{OT2}{cmr}{\hyphenchar\font45 }
                \DeclareFontShape{OT2}{cmr}{m}{n}{%
                     <5><6><7><8><9><10>gen*wncyr %
                     <10.95><12><14.4><17.28><20.74><24.88> wncyr10 %
                }{}
                \DeclareFontShape{OT2}{cmr}{m}{it}{%
                     <5><6><7><8><9><10> gen * wncyi%
                     <10.95><12><14.4><17.28><20.74><24.88> wncyi10%
                }{}
                \DeclareFontShape{OT2}{cmr}{bx}{n}{%
                     <5><6><7><8><9><10> gen * wncyb%
                     <10.95><12><14.4><17.28><20.74><24.88> wncyb10%
                }{}
                \DeclareFontShape{OT2}{cmr}{m}{sl}{%
                     <-> ssub * cmr/m/it%
                }{}
                \DeclareFontShape{OT2}{cmr}{m}{sc}{%
                     <5><6><7><8><9><10>%
                     <10.95><12><14.4><17.28><20.74><24.88> wncysc10%
                }{}
                \DeclareFontFamily{OT2}{cmss}{\hyphenchar\font45 }
                \DeclareFontShape{OT2}{cmss}{m}{n}{%
                     <8><9><10> gen * wncyss%
                     <10.95><12><14.4><17.28><20.74><24.88> wncyss10%
                }{}
                \def\cyr{\protect\pcyr}
                \def\cyrencodingdefault{OT2}
                \def\pcyr{%
                        \cyracc
                        \let\encodingdefault\cyrencodingdefault
                        \language\Russian\fontencoding{OT2}\selectfont
                }
             }

\@ifundefined{theorembodyfont}
     {
        \def\theorembodyfont#1{\relax}
        \makeatletter
          \let\@@th@plain\th@plain
          \def\th@plain{ \@@th@plain \slshape }
        \makeatother
        \let\normalshape\relax
     }{}

\ifx\cprime\undefined
     \def\cprime{$'$}
\fi
\makeatletter
  \def\@sect@my#1#2#3#4#5#6[#7]#8{%
\ifnum #2>\c@secnumdepth
   \let\@svsec\@empty
 \else
   \refstepcounter{#1}%
\edef\@svsec{\ifnum#2<\@m
             \@ifundefined{#1name}{}{\csname #1name\endcsname\ }\fi
\noexpand\rom{\csname the#1\endcsname.}\enspace}\fi
 \@tempskipa #5\relax
 \ifdim \@tempskipa>\z@ % then this is not a run-in section heading
   \begingroup #6\relax
   \@hangfrom{\hskip #3\relax\@svsec}{\interlinepenalty\@M #8\par}%
   \endgroup
   \if@article\else\csname #1mark\endcsname{%
        \ifnum \c@secnumdepth >#2\relax\csname the#1\endcsname. \fi#7}\fi
\ifnum#2>\@m \else
       \let\@tempf\\ \def\\{\protect\\}\addcontentsline{toc}{#1}%
{\ifnum #2>\c@secnumdepth \else
             \protect\numberline{%
               \ifnum#2<\@m
               \@ifundefined{#1name}{}{\csname #1name\endcsname\ }\fi
               \csname the#1\endcsname.}\fi
           #8}\let\\\@tempf
     \fi
 \else
  \def\@svsechd{#6\hskip #3\@svsec
    \@ifnotempty{#8}{\ignorespaces#8\unskip
       \ifnum\spacefactor<1001.\fi}%
        \ifnum#2>\@m \else
          \let\@tempf\\ \def\\{\protect\\}\addcontentsline{toc}{#1}%
            {\ifnum #2>\c@secnumdepth \else
              \protect\numberline{%
                \ifnum#2<\@m
                \@ifundefined{#1name}{}{\csname #1name\endcsname\ }\fi
                \csname the#1\endcsname.}\fi
             #8}\let\\\@tempf\fi}%
 \fi
\@xsect{#5}}
  \ifx\RequirePackage\undefined
  \let\@sect\@sect@my             % Cannot just comment the above
  \fi
  \ifx\RequirePackage\undefined %              Not under LaTeX2e
  \def\th@remark@my{\theorempreskipamount6\p@\@plus6\p@
    \theorempostskipamount\theorempreskipamount
    \def\theorem@headerfont{\it}\normalshape}
    \let\th@remark\th@remark@my
  \else             %              Under LaTeX2e
    \let\o@@remark\th@remark

    \ifx\theorempostskipamount\undefined        % AmS-art
    \else                       % Transformation groups
      \def\th@remark{\o@@remark
    \ifdim\theorempostskipamount < 2pt\relax
      \theorempostskipamount\theorempreskipamount
         \multiply\theorempostskipamount\tw@
         \divide\theorempostskipamount\thr@@
    \fi
      }
    \fi
  \fi
\let\myLabel\@gobble
\def\labelsONmargin{\@mparswitchfalse\def\myLabel##1{\@bsphack\marginpar
                                  {\normalshape\tiny\rm Label ##1}\@esphack}}
\ifx \url\undefined
  \def\url#1{{\tt #1}}%
\fi
\def\cyracc{\def\u##1{%\if \i##1\accent"24 i%
                \if \i##1\char"1A%
                \else \if I##1\char"12%
                \else \accent"24 ##1\fi\fi }%
\def\"##1{\if e##1{\char"1B}%
                \else \if E##1{\char"13}%
                \else \accent"7F ##1\fi\fi }%
\def\9##1{\if##1z\char"19
\else\if##1Z\char"11
\else\if##1E\char"03
\else\if##1e\char"0B
\else\if##1u\char"18
\else\if##1U\char"10
\else\if##1A\char"17
\else\if##1a\char"1F
\else\if##1p\char"7E
\else\if##1P\char"5E
\else\if##1Q\char"5F
\else\if##1q\char"7F
\else\if##1i\char"1A
\else\if##1I\char"12
\else\if##1N\char"7D
\fi
\fi
\fi
\fi
\fi
\fi
\fi
\fi
\fi
\fi
\fi
\fi
\fi
\fi
\fi
}%
\def\cydot{{\kern0pt}}}%
\def\cydot{$\cdot$}
\ifx\Russian\undefined
        \def\Russian{0\relax
    \message{Don't know the hyphenation rules for Russian^^J
                        Please do INITeX with `input  russhyph' in the
                        command line}%
                \gdef\Russian{0\relax}%
        }
\fi
\ifx\RequirePackage\undefined           % for 209
  \def\@putname#1#2#3#4{\def\@@ref{#3}\let\old@bf\bf
        \def\bf##1{\old@bf\if?\noexpand##1?{#4}\else##1\fi}%
    #1{#2}%
        \let\bf\old@bf}
\else                       % for 2e
  \def\@putname#1#2#3#4{\def\@@ref{#3}\let\old@bf\bf    % for 209???
    \let\old@reset@font\reset@font          % for 2e
        \def\bf##1{\old@bf\if?\noexpand##1?{#4}\else##1\fi}%
    \def\reset@font##1##2{\old@reset@font##1\if?\noexpand##2?{#4}\else##2\fi}#1{#2}%
        \let\bf\old@bf\let\reset@font\old@reset@font}
\fi
\let\my@ref=\ref
\def\ref#1{\@putname\my@ref{#1}{#1}{\tiny\rm\@@ref}}
\let\my@pageref=\pageref
\def\pageref#1{\@putname\my@pageref{#1}{#1}{\tiny\rm\@@ref}}
\let\my@cite=\cite
\def\cite#1{\@putname\my@cite{#1}{\@citeb}{\tiny\rm\@@ref}}
\makeatother
\fi
\relax
\theoremstyle{plain} % for references in unnumbered theorems
\theorembodyfont{\sl}

\ifx\useDoubleSpacing\undefined
\else
  \usepackage{setspace}
\fi

\ifx \numberingIsThrough \undefined
\numberwithin{equation}{section}
\fi
\input xy
\xyoption{all}
\theorembodyfont{\rm}
\theoremstyle{definition}

\ifx \numberingIsThrough \undefined
\newtheorem{definition}{Definition}[section]
\else
\newtheorem{definition}{Definition}
\fi

\newtheorem{example}[definition]{Example}

\theoremstyle{remark}

\newtheorem{remark}[definition]{Remark} %\renewcommand{\theremark}{}

\ifx \numberingIsThrough \undefined
\newtheorem{note}{Note}[section] 
\newtheorem{summary}{Summary}[section] 
\else

\fi

\theoremstyle{plain} % for future references
\theorembodyfont{\sl}

\newtheorem{theorem}[definition]{Theorem}
\newtheorem{lemma}[definition]{Lemma}
\newtheorem{corollary}[definition]{Corollary}
\newtheorem{proposition}[definition]{Proposition}

\begin{document}
\bibliographystyle{amsplain}

\ifx\useHugeSize\undefined
\else
\Huge
\fi

\relax

\title{ Cuspidal representations of $\mathfrak{sl}\left(n+1\right) $ }

\author{ Dimitar Grantcharov and Vera Serganova }

\date{ \today }

\address{ Dept. of Mathematics, University of Texas at Arlington,
Arlington, TX 76019 } \email{grandim\atSign{}uta.edu}

\address{ Dept. of Mathematics, University of California at Berkeley,
Berkeley, CA 94720 } \email{serganov\atSign{}math.berkeley.edu}

\maketitle

\begin{abstract}
In this paper we study the subcategory of cuspidal modules of the
category of weight modules over the Lie algebra
$\mathfrak{sl}(n+1)$. Our main result is a complete classification
and an explicit description of indecomposable cuspidal modules.\\

\noindent 2000 MSC: 17B10\\

\noindent \keywords{Keywords and phrases}: Lie algebra,
indecomposable representations, quiver, weight modules
\end{abstract}

\section{Introduction}
The category of weight representations has attracted considerable
mathematical attention in the last thirty years. General weight
modules have been extensively studied by  G. Benkart, D. Britten,
S. Fernando, V. Futorny, A. Joseph, F. Lemire, and others (see
e.g. \cite{BBL}, \cite{BL1}, \cite{BL2}, \cite{Fe}, \cite{Fu}).
Following their works, in 2000 O. Mathieu, \cite{M}, established
the classification of all simple weight modules with
finite-dimensional weight spaces over reductive Lie algebras.
An important role in this classification plays the category
${\mathcal C}$ of all cuspidal modules, i.e. weight modules on
which all root vectors of the Lie algebra act bijectively. This
role is due to the parabolic induction theorem of Fernando and
Futorny. The theorem states that every simple weight module $M$ with
finite-dimensional weight spaces over a reductive
finite-dimensional Lie algebra ${\mathfrak g}$ is isomorphic to the unique simple
quotient of a parabolically induced
generalized Verma module $U(\mathfrak g) \otimes_{U({\mathfrak p})}S$, where $S$ is
a cuspidal module over the Levi component of ${\mathfrak p}$. The
Fernando--Futorny result naturally initiates the study of the
category ${\mathcal C}$ as a necessary step towards the study of
all weight modules with finite weight multiplicities.

Let $\mathfrak g$ be a simple complex finite-dimensional Lie
algebra. A result of Fernando implies that ${\mathcal C}$ is nontrivial
 for Lie algebras $\mathfrak g$
of type $A$ and $C$ only. In the symplectic case the category
$\mathcal C$ is semisimple (see \cite{BKLM}). In the present paper
we focus on the  remaining case, i.e.  $\mathfrak{g} =
\mathfrak{sl} (n+1)$. The main result is a classification of the
indecomposable modules in this case.

There are two major differences between the two algebra types that
make the study of the cuspidal modules much harder in the $A$-type
case. First, in the $\mathfrak{sl}$-case, the translation functor
does not provide equivalence of the subcategories ${\mathcal
C}^{\chi}$ of ${\mathcal C}$ for all central characters $\chi$.
This leads to a consideration of three essentially different
central character types: nonintegral, regular integral and
singular. Second, in many cases a simple cuspidal module has a
nontrivial self-extension. Because of this the category of
cuspidal modules does not have projective and injective objects,
and one has to use a certain completion $\overline{\mathcal C}$ of ${\mathcal C}$. The most interesting
central character type is, without a doubt, the regular integral
one, because in this case there are $n$ up to isomorphism simple
objects in every block  of ${\mathcal C}^{\chi}$. A convenient way to approach this case is
to use methods and results from the quiver theory.
 Our main result can be formulated as follows

\begin{theorem}
(a) Every singular and nonintegral block of the category ${\mathcal C}$ is
equivalent to the category of finite-dimensional modules over the
algebra of power series in one variable.

(b) Every regular integral block of ${\mathcal C}$  is equivalent to the category
of locally nilpotent modules over the quiver
$$
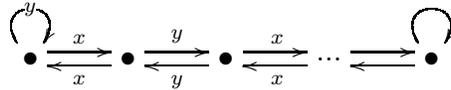
\xymatrix{\bullet \ar@(ul,ur)[]|{y} \ar@<0.5ex>[r]^x & \bullet
\ar@<0.5ex>[l]^x \ar@<0.5ex>[r]^y & \bullet \ar@<0.5ex>[l]^y
\ar@<0.5ex>[r]^x & \ar@<0.5ex>[l]^x... \ar@<0.5ex>[r] & \bullet
\ar@<0.5ex>[l] \ar@(ul,ur)[]|{}}
$$
with relations $xy=yx=0$.
\end{theorem}

%In case of non-integral or singular integral central character $\chi$
%each block of the category ${\mathcal C}^{\chi}$ has just one
%simple object and the problem of describing the indecomposable
%modules is equivalent to the classification of all indecomposable
%representation of a nilpotent operator. This somehow
%surprising correspondence helps us to understand completely the
%structure of the blocks ${\mathcal C}^{\chi}_{\bar{\nu}}$.
The above quiver is special biserial and hence tame (see \cite{Er}). It was
originally studied by
Gelfand--Ponomarev in \cite{GP} in order to classify
indecomposable representations of the Lorentz group.

Another important aspect of the category ${\mathcal C}$ is the
geometric realization of its objects. The simple objects in
${\mathcal C}$, as well as their injective hulls in
${\bar{\mathcal C}}$, can be realized with the
aid of sections of vector bundles on the projective space
${\mathbb P}^n$. These realizations are especially helpful for the
explicit calculations of extensions of simple modules in the category ${\mathcal C}$.

The results in the present paper make a first step towards
the study of other interesting category: the category ${\mathcal
B}$ of all weight $\mathfrak g$-modules with uniformly bounded
weight multiplicities. The study of the category ${\mathcal B}$
was initiated in \cite{GS}, where the case of ${\mathfrak g} =
{\mathfrak s} {\mathfrak p} (2n)$ was completely solved. The
remaining case, i.e. ${\mathfrak g} = \mathfrak{sl}(n+1)$, will be
treated in a future work.

Another natural problem is to extend the study of cuspidal modules to the
 category of generalized weight modules, i.e. modules that decompose as direct sums of  generalized weight spaces on which $h - \lambda(h)$ act locally finitely for every weight $\lambda$ and $h$ in  the Cartan subalgebra
 of $\mathfrak g$. The category of generalized cuspidal modules have the
 same simple objects as the category of simple weight modules, but the
 indecomposables are different. In fact, it follows from Theorem
 \ref{genthm} that the category of generalized cuspidal modules is wild,
 so it should be studied using other methods. On the other hand,
the geometric constructions we obtain in this paper can be
easily generalized and lead to examples of generalized cuspidal modules
that are generalized weight modules but not weight modules (see Remark \ref{generalized}).

The organization of the paper is as follows. In Section 2
we introduce the main notions and with the aid of the translation
functor, reduce the general central character case to a specific
set of central characters (namely, those that correspond to
multiples of $\varepsilon_0$). In Section 3 we prove some
preparatory statements and, in particular, consider
the case of ${\mathfrak s} {\mathfrak l} (2)$. Let us mention
that the classification of the indecomposable
weight modules over ${\mathfrak s} {\mathfrak l} (2)$ with scalar
action of Casimir is usually attributed to Gabriel \cite{Gab1} and can
be found in \cite{Dix},  \S 7.8.16. The general case was treated in
\cite{Dr}.  In Section 4 we calculate extensions between simple
cuspidal modules with non-integral and singular central
characters. In Section 5 we extend the category $\mathcal C$ of
cuspidal modules to
$\bar{\mathcal C}$ by adding injective limits, and construct injective
objects in blocks with non-integral and singular central
characters. The case of a regular integral central character is
treated in Section 6, where we use translation functors (\cite{BG})
from singular to regular blocks to construct injective modules in the
regular case. In the last section we
provide an explicit realization of all indecomposable cuspidal
modules.
We expect that the description of the indecomposables of ${\mathcal C}$ will be useful for studying other categories of weight modules, including the category ${\mathcal B}$ of bounded modules.

\smallskip
\noindent{\bf Acknowledgements.} We would like to thank K.
Erdmann, S. Koenig, and C. Ringel for providing us with very
useful information about certain quivers. We are thankful to I.
Zakharevich for reading the first version of the manuscript and
providing valuable comments. Special thanks are due to the referee for
the very careful reading of the manuscript and providing numerous suggestions that substantially improved the paper.

\section{Cuspidal representations }

In this paper the ground field is $\mathbb C$, and ${\mathfrak
g}=\mathfrak{sl}\left(n+1\right) $. All tensor
products are assumed to be over $\mathbb C$ unless otherwise
stated. We fix a Cartan subalgebra $ {\mathfrak h}$ of ${\mathfrak
g} $ and denote by $(\, , )$ the Killing form on
$\mathfrak g$. The induced form on ${\mathfrak h}^*$ will be
denoted by $(\, ,)$ as well. For our convenience, we fix a
 basis $ \{ \varepsilon_{0},\dots ,\varepsilon_{n} \}$ in $ {\mathbb
C}^{n+1} $, such that $ {\mathfrak h}^{*} $ is identified with the
subspace of $ {\mathbb C}^{n+1} $ spanned by the simple roots $
\alpha_{1}=\varepsilon_{0}-\varepsilon_{1},\dots
,\alpha_{n}=\varepsilon_{n-1}-\varepsilon_{n} $.  By $ \gamma $ we
denote the projection $ {\mathbb C}^{n+1} \to
{\mathfrak h}^{*} $ with one-dimensional kernel
$\mathbb C(\varepsilon_0+\dots+\varepsilon_n)$.
By $ Q \subset {\mathfrak h}^{*}$ and $
\Lambda \subset {\mathfrak h}^{*}$ we denote the root lattice and
the weight lattice of $ {\mathfrak g} $, respectively. The basis $
\left\{\omega_{1},\dots ,\omega_{n}\right\} $ of $ \Lambda $
consists of the fundamental weights $\omega_i :=
\gamma(\varepsilon_0+...+\varepsilon_{i-1})$ for $i=1,...,n$. Let
$U:=U(\mathfrak{g})$ be the universal enveloping algebra of
$\mathfrak g$. Denote by $Z:=Z(U(\mathfrak{g}))$ the center of $U$
and let $Z':=\mbox{Hom} (Z, \mathbb C)$ be the set of all central
chatacters (here $\mbox{Hom}$ stands for
homomorphisms of unital $\mathbb C$-algebras). By $\chi_{\lambda}\in Z'$ we
denote the central character of the irreducible highest weight
module with highest weight $\lambda$. Recall that
$\chi_{\lambda}=\chi_{\mu}$ iff $\lambda+\rho=w(\mu+\rho)$ for
some element $w$ of the Weyl group $W$, where, as usual, $\rho$
denotes the half-sum of positive roots. We say that
$\chi=\chi_{\lambda}$ is {\it regular} if the stabilizer of
$\lambda+\rho$ in $W$ is trivial (otherwise $\chi$ is called {\it
singular}), and that $\chi=\chi_{\lambda}$ is {\it integral} if
$\lambda\in\Lambda$. We say that two weights $\lambda$ and $\nu\in
\lambda+\Lambda$ are in the same Weyl chamber if for any positive
root $\alpha$ such that $(\lambda,\alpha)\in \mathbb Z$,
$(\lambda,\alpha)\in \mathbb Z_{\geq 0}$ if and only if
$(\mu,\alpha)\in \mathbb Z_{\geq
  0}$. Finally recall that $\lambda$ is {\it dominant integral} if
$(\lambda,\alpha)\in \mathbb Z_{\geq 0}$ for all positive roots $\alpha$.

A $ {\mathfrak g} $-module $ M $ is a {\it generalized weight module} if $
M=\bigoplus_{\mu \in {\mathfrak h}^*} M^{(\mu)} $, where
\begin{equation}
M^{(\mu)}:=\left\{m\in M \mid \left( h - \mu \, {\rm Id} \right)^N m = 0, \mbox{ for
every } h\in{\mathfrak h} \mbox{, and
some } N = N(h,m) \right\}. \notag
\end{equation}

A generalized weight $ {\mathfrak g} $-module $ M $ is called a {\it weight module} if $
M^{(\mu)}  =  M^{\mu} $, where
\begin{equation}
M^{\mu}:=\left\{m\in M \mid hm=\mu\left(h\right)m, \mbox{ for
every } h\in{\mathfrak h}\right\}. \notag\end{equation} By
definition the {\it support of $M$},  $ \operatorname{supp} M $,
is the set of weights $ \mu\in{\mathfrak h}^{*} $ such that $
M^{\mu}\not=0 $. A weight $\mathfrak g$-module $ M $ is {\em
cuspidal\/} if $ M $ is finitely generated, all $ M^{\mu} $ are
finite-dimensional, and $ X:M^{\mu} \to M^{\mu+\alpha} $ is an
isomorphism for every root vector $ X\in{\mathfrak g}_{\alpha} $.
Denote by ${\mathcal C}$ the category of all cuspidal $\mathfrak
g$-modules.

It is clear that, if $M$ is a cuspidal module, then  $
\mu\in\operatorname{supp} M $ implies $
\mu+Q\subset\operatorname{supp} M $. Hence for every cuspidal
module $ M $ one can define $ s\left(M\right)\subset{\mathfrak
h}^{*}/Q $ as the image of $ \operatorname{supp} M $ under the
natural projection $ {\mathfrak h}^{*} \to {\mathfrak h}^{*}/Q $.
As $ M $ is finitely generated, $ s\left(M\right) $ is a finite
set.

It is not difficult to see that a submodule and a quotient of a
cuspidal module are cuspidal. Hence the category $ {\mathcal C} $
is an abelian category. It is also clear that every cuspidal
module has finite Jordan--H\"older series. Since the center $Z$ of $U$
preserves weight spaces, it acts locally finitely on the cuspidal
modules. For every central character $ \chi\in Z'$ let $ {\mathcal
C}^{\chi} $ denote the category of all cuspidal modules $ M $ with
generalized central character $\chi$, i.e. such that for some $n\left(M\right),
\left(z-\chi\left(z\right)\right)^{n\left(M\right)}=0 $ on $M$ for
all $ z\in Z$. It is clear that every cuspidal module $ M $ is a
direct sum of finitely many $ M_{i}\in{\mathcal C}^{\chi_{i}} $.
Furthermore, if $ \operatorname{Ext}_{\mathcal C}$ stands for
the extension functor in the category $\mathcal C$, then
$ \operatorname{Ext}_{\mathcal C} \left(M,N\right)\not=0 $ implies
$ s\left(M\right)\cap s\left(N\right)\not=\varnothing $. Therefore
every $ M\in{\mathcal C}^{\chi} $ is a direct sum $
M=\bigoplus_{\bar{\nu}\in s\left(M\right)}M[\bar{{\nu}}] $, where,
for $\bar{\nu}:=\nu + Q \in {\mathfrak h}^*/Q$,  $ M[\bar{{\nu}}]$
denotes the maximal submodule of $ M $ such that $
s\left(M[{\bar{\nu}}]\right)=\left\{\bar{\nu}\right\} $. Thus, one
can write
\begin{equation}
{\mathcal C}=\bigoplus_{\chi \in Z',\; \bar{\nu} \in {\mathfrak
h}^*/Q}{\mathcal C}^{\chi}_{\bar{\nu}} \notag\end{equation} where
$ {\mathcal C}^{\chi}_{\bar{\nu}} $ is the category of all modules
$ M $ in $ {\mathcal C}^{\chi} $ such that $
s\left(M\right)=\left\{\bar{\nu}\right\} $.

In this section we describe the simple cuspidal $ {\mathfrak
g}-$modules following the classification of Mathieu in \cite{M}.
We formulate Mathieu's result in convenient for us terms.

We fix an ${\mathfrak h}$-eigenbasis of the natural representation
of $\mathfrak g$. Let $ E_{ij},i,j=0,\dots ,n $
denote the elementary matrices. Define a $ {\mathbb Z} $-grading $
{\mathfrak g}={\mathfrak g}_{-1}\oplus{\mathfrak
g}_{0}\oplus{\mathfrak g}_{1} $, where $ {\mathfrak g}_{1} $ is
spanned by $ E_{i0} $ for all $ i>0 $, $ {\mathfrak g}_{-1} $ is
spanned by $ E_{0i} $ for all $ i>0 $, and $ {\mathfrak
g}_{0}\cong  \mathfrak{gl}\left(n\right) $. The subalgebra $
{\mathfrak p}={\mathfrak g}_{0}\oplus{\mathfrak g}_{1} $ is a
maximal parabolic subalgebra in $ {\mathfrak g} $. If $
G=\operatorname{SL}\left(n+1\right) $ and $ P\subset G $ is the
subgroup with the Lie algebra $ {\mathfrak p} $, then $ G/P $ is
isomorphic to $ {\mathbb P}^{n} $. Every point $t$ in  $ {\mathbb
P}^{n}$ can be represented by its homogeneous coordinates
$\left[t_{0},\dots ,t_{n}\right]$. Then $ {\mathfrak g} $ defines
an algebra of vector fields on $ {\mathbb P}^{n} $ via the map
\begin{equation}
E_{ij} \mapsto t_{i}\frac{\partial}{\partial t_{j}}.
\label{equ1}\end{equation}\myLabel{equ1,}\relax By $ E $ we denote
the Euler vector field $
\sum_{i=0}^{n}t_{i}\frac{\partial}{\partial t_{i}} $.

Let $\mathcal {U}$  be the affine open subset of ${\mathbb P}^{n}
$ consisting of all points $[t_0,...,t_n]$ such that $ t_{0}\not=0
$. Introduce local coordinates $ x_{1},\dots ,x_{n} $ on $
\mathcal {U} $ by setting $ x_{i}:=\frac{t_{i}}{t_{0}} $ for all $
i=1,\dots ,n $. The ring $ {\mathcal O}={\mathbb
C}\left[x_{1},\dots ,x_{n}\right] $ of regular functions on $
\mathcal {U} $ is naturally a $ {\mathfrak g} $-module. We say
that $ M $ is a $ \left({\mathfrak g},{\mathcal O}\right) ${\it
-module} if $ M $ is both a $ {\mathfrak g} $-module and an
${\mathcal O} $-module, and
\begin{equation}
g\left(f m\right)-f\left(g m\right)=g\left(f\right)m
\notag\end{equation} for any $ m\in M $, $ g\in{\mathfrak g} $, $
f\in {\mathcal O} $. In particular, $ {\mathcal O} $ is a $ \left({\mathfrak
g},{\mathcal O}\right) $-module. If $ M $ and $ N $ are $
\left({\mathfrak g},{\mathcal O}\right) $-modules, then $
M\otimes_{{\mathcal O}}N $ is a $ ({\mathfrak g}, {\mathcal O})
$-module as well.

For $ \mu\in{\mathbb C}^{n+1} $, let $ |\mu| $ denote $
\mu_{0}+\dots +\mu_{n} $ and $ t^{\mu} $ denote $
t_{0}^{\mu_{0}}\dots t_{n}^{\mu_{n}} $. Then the vector space $
t^{\mu}{\mathbb C}\left[t_{0}^{\pm1},\dots ,t_{n}^{\pm1}\right] $
has a natural structure of a $ \left({\mathfrak g},{\mathcal
O}\right) $-module. Let
\begin{equation}
{\mathcal F}_{\mu}:=\left\{f\in t^{\mu}{\mathbb
C}\left[t_{0}^{\pm1},\dots ,t_{n}^{\pm1}\right] \mid
Ef=|\mu|f\right\}. \notag\end{equation} One can check that $
{\mathcal F}_{\mu} $ is a $ \left({\mathfrak g},{\mathcal
O}\right) $-module. It is not hard to see that $ {\mathcal
F}_{\mu} $ is an irreducible cuspidal $ {\mathfrak g} $-module iff
$ \mu_{i}\notin{\mathbb Z} $ for all $ i=0,\dots ,n $, in that case
all weight spaces are one-dimensional. It is also
clear from the construction that $ {\mathcal F}_{\mu}\cong
{\mathcal F}_{\mu'} $ if $ \mu-\mu'\in Q $.

Now we recall a slightly more general construction of a cuspidal
module (see \cite{Sh}).
Let $ V_{0} $ be a finite-dimensional $ P $-module. Then $ V_{0} $
induces a vector bundle on $ {\mathbb P}^{n} $ which we denote by
$ {\mathcal V}_{0} $. The space of its sections $
\Gamma\left(\mathcal {U},{\mathcal V}_{0}\right) $ on the affine
open set $ \mathcal {U}\subset G/P $ is another example of a $
\left({\mathfrak g},{\mathcal O}\right) $-module. Since $
{\mathcal V}_{0} $ is trivial on $ \mathcal {U} $ one can identify
{\cyr G}($ \mathcal {U},{\mathcal V}_{0} $) with $ {\mathcal
O}\otimes V_{0} $. Define a new $ {\mathfrak g} $-module
\begin{equation}
{\mathcal F}_{\mu}\left(V_{0}\right)={\mathcal
F}_{\mu}\otimes_{{\mathcal O}}\Gamma\left(\mathcal {U},{\mathcal
V}_{0}\right). \notag\end{equation} Again, one easily checks that
$ {\mathcal F}_{\mu}\left(V_{0}\right) $ is cuspidal iff $
\mu_{i}\notin{\mathbb Z} $.

\begin{remark} \label{rem99}\myLabel{rem99}\relax  If $ V $ is a
$ {\mathfrak g} $-module, then $ \Gamma\left(\mathcal {U},{\mathcal
V}\right)\cong {\mathcal O}\otimes V $ as a $\mathfrak g$-module, and
we have the following isomorphisms of $\mathfrak g$-modules
\begin{eqnarray*}
{\mathcal F}_{\mu}\left(V_{0}\right)\otimes V &\cong & {\mathcal
F}_{\mu}\left(V_{0}\right)\otimes_{{\mathcal
O}}\Gamma\left(\mathcal {U},{\mathcal V}\right) \cong {\mathcal
F}_{\mu}\otimes_{{\mathcal O}}\Gamma\left(\mathcal {U},{\mathcal
V}_{0}\right)\otimes_{{\mathcal O}}\Gamma\left(\mathcal {U},{\mathcal
V}\right)\\ &\cong  &{\mathcal F}_{\mu}\otimes_{{\mathcal
O}}\Gamma\left(\mathcal {U},{\mathcal V}\otimes{\mathcal V}_{0}\right)\cong
{\mathcal F}_{\mu}\left(V_{0}\otimes V\right). \notag
\end{eqnarray*}
\end{remark}

\begin{remark} \label{rem100}\myLabel{rem100}\relax  An exact sequence of $ P $-modules
\begin{equation}
0 \to V_{0} \to V_{1} \to V_{2} \to 0
\notag\end{equation}
induces the exact sequence of $ {\mathfrak g} $-modules
\begin{equation}
0 \to {\mathcal F}_{\mu}\left(V_{0}\right) \to {\mathcal F}_{\mu}\left(V_{1}\right) \to {\mathcal F}_{\mu}\left(V_{2}\right) \to\text{ 0.}
\notag\end{equation}
\end{remark}

Recall that  $ \chi_{\lambda} $ denotes the central character of the simple
highest weight $\mathfrak g$-module with highest weight $ \lambda $. For
simplicity we put $ {\mathcal C}^{\lambda}:={\mathcal
C}^{\chi_{\lambda}} $ and $ {\mathcal
C}^{\lambda}_{\bar{\nu}}:={\mathcal
C}^{\chi_{\lambda}}_{\bar{\nu}}$.

\begin{lemma} \label{lm1}\myLabel{lm1}\relax
Let $ V_{0} $ be an irreducible $ P $-module with highest weight $ \tau $.
Then $ {\mathcal F}_{\mu}\left(V_{0}\right) $ admits a central
character $ \chi_{\lambda} $, where $
\lambda=\gamma\left(|\mu|\varepsilon_{0}\right)+\tau $. Moreover,
$ s\left({\mathcal F}_{\mu}\left(V_{0}\right)\right) $ coincides
with the class of $ \gamma\left(\mu\right)+\tau $ in $ {\mathfrak
h}^{*}/Q $.

\end{lemma}

\begin{proof} As a $\mathfrak g_0$-module $ {\mathcal
    F}_{\mu}\left(V_{0}\right) $ is isomorphic to
\begin{equation}\label{restr}
\mathcal F_{\mu}\otimes V_0=x_{1}^{\mu_{1}}\dots x_{n}^{\mu_{n}}{\mathbb
C}\left[x_{1}^{\pm1},\dots ,x_{n}^{\pm1}\right]\otimes V_{0}.
\notag
\end{equation}
We will  consider $ {\mathfrak g} $ as a Lie subalgebra of $
{\mathcal A}_{n}\otimes\operatorname{End}\left(V_{0}\right) $,
where $ {\mathcal A}_{n} $ is the algebra of polynomial
differential operators in $ {\mathbb C}\left[x_{1},\dots
,x_{n}\right] $. The embedding $ {\mathfrak g} \to {\mathcal
A}_{n}\otimes\operatorname{End}\left(V_{0}\right) $ is defined by
the formulae
\begin{equation}
E_{0i} \mapsto \frac{\partial}{\partial x_{i}}\text{, }E_{ij}
\mapsto x_{i}\frac{\partial}{\partial x_{j}}\otimes1+1\otimes
E_{ij}\text{, }i,j>0,i\not=j, \notag\end{equation}
\begin{equation}
E_{00}-E_{i i} \mapsto
\left(-|\mu|-E'-x_{i}\frac{\partial}{\partial
x_{i}}\right)\otimes1+1\otimes\left(E_{00}-E_{i i}\right)\text{,
}i>0, \notag\end{equation}
\begin{equation}
E_{i0} \mapsto
\left(-x_{i}\left(E'+|\mu|\right)\right)\otimes1-\sum_{j\not=i}x_{j}\otimes
E_{ij}+x_{i}\otimes\left(E_{00}-E_{i i}\right),
\notag\end{equation} where $E':=\sum_{i=1}^n x_i\frac{\partial}
{\partial x_i}$. This embedding can be extended to a homomorphism $
\chi_{\mu,\tau}:U\left({\mathfrak g}\right) \to {\mathcal
A}_{n}\otimes\operatorname{End}\left(V_{0}\right) $. Note now that
the formulae defining $ \chi_{\mu,\tau} $ depend only on $ |\mu| $
and $ \tau $. Hence $\chi_{\nu,\tau}=\chi_{\mu,\tau} $
 if $ |\nu|=|\mu| $. Let $ \nu=|\mu|\varepsilon_{0}
$. Then $ {\mathcal F}_{\nu}\left(V_{0}\right) $ is not cuspidal and
contains a ${\mathfrak b} $-singular vector $ t_{0}^{|\mu|}\otimes v_{0} $,
where $ v_{0} $ is a highest weight vector of $ V_{0} $ and
$\mathfrak b$ is the standard Borel subalgebra of $\mathfrak g$.
Hence $ {\mathcal F}_{\nu}\left(V_{0}\right) $ contains a
submodule $ L $ with the highest weight $
\gamma\left(\nu\right)+\tau=\lambda $. It is not difficult to see
that $ L $ is a faithful module over $ {\mathcal
A}_{n}\otimes\operatorname{End}\left(V_{0}\right) $. Since $ L $
admits a central character $ \chi_{\lambda} $,  $
\chi_{\mu,\tau}\left(z\right)=\chi_{\nu,\tau}\left(z\right)=\chi_{\lambda}\left(z\right)
$ for any $ z $ in $Z$. The second statement of the lemma is
trivial.\end{proof}

The formulae in the proof of the previous lemma depend only on $|\mu|$
and the representation $V_0$ of
$\left[{\mathfrak g}_{0},{\mathfrak
    g}_{0}\right]\simeq \mathfrak {sl}(n)\subset \mathfrak g$.
Therefore we obtain the following corollary.

\begin{corollary} \label{cor1}\myLabel{cor1}\relax  If $ \left(\tau,\alpha_{i}\right)=\left(\tau',\alpha_{i}\right) $ for $ i=2,\dots ,n $ and
$ \mu'-\mu\in Q+\left(\tau-\tau',\alpha_{1}\right)\varepsilon_{0} $, then $ {\mathcal F}_{\mu}\left(V_{0}\right) $ is
isomorphic to $ {\mathcal F}_{\mu'}\left(V'_{0}\right) $, where $ \tau $ and $ \tau' $ are the highest weights of $ V_{0} $ and
$ V'_{0} $ respectively.

\end{corollary}

The following theorem is proven in \cite{M}.

\begin{theorem} \label{th1}\myLabel{th1}\relax  Suppose that $ M $ is an irreducible cuspidal $ {\mathfrak g} $-module and
$ \chi $ be the central character of $ M $. Then $
\chi=\chi_{\lambda} $ where $ \lambda $ is integral dominant when
restricted to the subalgebra
$ \left[{\mathfrak g}_{0},{\mathfrak g}_{0}\right] $.
If $ \chi $ is non-integral or singular integral, then $ M $ is
isomorphic to $ {\mathcal F}_{\mu}\left(V_{0}\right) $ for some $
\mu\in{\mathbb C}^{n+1} $ and an irreducible $ P$-module $
V_{0} $.

\end{theorem}

This theorem gives a complete description of cuspidal simple
modules with singular or non-integral central characters. We
provide a description of the simple cuspidal modules with regular
integral central character (also obtained by Mathieu) in Section
\ref{regint}.

Theorem~\ref{th1} implies:

\begin{corollary} \label{cor2}\myLabel{cor2}\relax  Let $\rho$ be the half-sum of the positive roots and
$ \lambda+\rho $ be a non-integral or a singular integral weight.
Then $ {\mathcal C}^{\lambda}_{\bar{\nu}} $ is not empty iff $
\left(\lambda,\omega_{n}\right)\notin\left(\nu,\omega_{n}\right)+{\mathbb
Z},
\left(\lambda,\omega_{i}-\omega_{i+1}\right)\notin\left(\nu,\omega_{i}-\omega_{i+1}\right)+{\mathbb
Z} $ for $ i=1,\dots ,n-1 $, and $
\left(\lambda,\omega_{1}-\omega_{2}\right)\notin -
\left(\nu,\omega_{1}\right)+{\mathbb Z} $. If $ {\mathcal
C}^{\lambda}_{\bar{\nu}}$ is not empty it has exactly one up to an
isomorphism simple object.

\end{corollary}

\begin{proof} Recall that $ \lambda=\gamma\left(|\mu|\varepsilon_{0}\right)+\tau $, where $ \tau $ is the highest weight of $ V_{0} $. As
follows from Corollary~\ref{cor1}, we can assume that $
\left(\tau,\alpha_{1}\right)=0 $. We fix $\mu \in {\mathfrak h}^*$
so that $ \varphi:=\gamma\left(\mu\right)+\tau $ is a
representative of $\bar{\nu}$ in ${\mathfrak h}^*/Q$, i.e.
$\bar{\varphi} = \bar{\nu}$. One has the following relations
between $ \mu,\lambda $ and $ \varphi $
\begin{equation}
\left(\varphi,\alpha_{1}\right)=\mu_{0}-\mu_{1}\text{,
}\left(\varphi,\alpha_{i}\right)=\mu_{i-1}-\mu_{i}+\left(\lambda,\alpha_{i}\right),
i=2,\dots ,n\text{, }\left(\lambda,\alpha_{1}\right)=|\mu|.
\notag\end{equation}

The above relations imply that
\begin{equation}
\mu_{n}\left(n+1\right)+\left(\varphi,n\alpha_{n}+\dots +
2\alpha_2 + \alpha_{1}\right)-\left(\lambda,n\alpha_{n}+\dots ++
2\alpha_2 + \alpha_{1}\right)=0. \notag\end{equation} Using that $
n\alpha_{n}+\dots + 2\alpha_2 +
 \alpha_{1}=\left(n+1\right)\omega_{n} $, we obtain
\begin{equation}
\mu_{n}=\left(\lambda-\varphi,\omega_{n}\right)\text{, }\mu_{n-1}=\left(\lambda-\varphi,\omega_{n-1}-\omega_{n}\right),\dots ,
\notag\end{equation}
\begin{equation}
\mu_{1}=\left(\lambda-\varphi,\omega_{1}-\omega_{2}\right),\mu_{0}=\left(\lambda,\omega_{1}-\omega_{2}\right)+\left(\varphi,\omega_{1}\right).
\notag\end{equation} Since $ {\mathcal F}_{\mu}\left(V_{0}\right)
$ is cuspidal iff $ \mu_{i}\notin{\mathbb Z} $ for all $ i=0,\dots
,n $, the proof is completed.\end{proof}

Recall now the definition of translation functors. Let $ V $ be a
finite-dimensional $ {\mathfrak g} $-module and $
\eta,\lambda\in{\mathfrak h}^{*} $ be such that $\tau = \lambda
-\eta$ is in the support of $V$. Let $ {\mathfrak g}^{\kappa}$-mod
denote the category of $ {\mathfrak g} $-modules which admit a
generalized central character $ \chi_{\kappa} $. The translation
functor $ T_{\eta}^{\lambda}:{\mathfrak g}^{\eta}\mbox{-mod} \to
{\mathfrak g}^{\lambda}\mbox{-mod}$ is defined by $
T_{\eta}^{\lambda}\left(M\right)=\left(M\otimes
V\right)^{\chi_\lambda} $, where $ \left(M\otimes
V\right)^{\chi_\lambda} $ stands for the direct summand of $
M\otimes V $ admitting generalized central character $
\chi_{\lambda} $. Assume in addition that $ \tau $ belongs to the
$W$-orbit of the highest weight of $ V $,
the stabilizers of $ \eta+\rho $ and $
\lambda+\rho $ in the Weyl group coincide and $ \nu+\rho $, $
\lambda+\rho $ lie in the same Weyl chamber. Then $
T_{\eta}^{\lambda}:{\mathfrak g}^{\eta}\mbox{-mod} \to {\mathfrak
g}^{\lambda}\mbox{-mod}$ defines an equivalence of categories (see
\cite{BG}).

\begin{lemma} \label{lm3}\myLabel{lm3}\relax  With the above notations, if $ M $ is cuspidal, then $ T_{\eta}^{\nu}\left(M\right) $
is cuspidal and $
s\left(M\right)+s\left(V\right)=s\left(T_{\eta}^{\nu}\left(M\right)\right)
$.

\end{lemma}

\begin{proof} First we observe that $ M\otimes V $ is a semisimple  $ {\mathfrak h}$-module with finite weight
multiplicities. Let $ X\in{\mathfrak g}_{\alpha} $. Since $ M $ is
cuspidal the action of $ X $ is free, therefore the action of $ X
$ on $ M\otimes V $ is free as well. Since the multiplicities of
all weight spaces in $ M $ are the same, the weight multiplicities
of $ M\otimes V $ are all the same as well. Therefore $ M\otimes V
$ is cuspidal. Since $ \left(M\otimes V\right)^{\chi_{\nu}} $ is a
direct summand of $ M\otimes V $, $ \left(M\otimes
V\right)^{\chi_{\nu}} $ is also cuspidal. The second statement is
obvious.\end{proof}

\begin{lemma} \label{lm2}\myLabel{lm2}\relax  If the category $ {\mathcal C}^{\chi} $ is not empty, then it is equivalent to
$ {\mathcal C}^{\gamma\left(t\varepsilon_{0}\right)} $ for some
$ t\in ({\mathbb C}\backslash{\mathbb Z})\cup \{0,-1,\dots,-n\} $.

\end{lemma}

\begin{proof} Since $ {\mathcal C}^{\chi} $ is not empty, Theorem~\ref{th1} implies that $ \chi=\chi_{\lambda} $ where $ \lambda $ is
dominant when restricted to $ \left[{\mathfrak g}_{0},{\mathfrak
g}_{0}\right] \cong  \mathfrak{sl}\left(n\right) $. If $ \lambda $
itself is integral dominant it is well known that $ {\mathfrak
g}^{\lambda}$-mod is equivalent to $ {\mathfrak g}^{0}\mbox{-mod}
$ (see \cite{BG} for instance). Thus, by Lemma \ref{lm3},
$ {\mathcal C}^{\lambda} $ is equivalent to $ {\mathcal C}^{0}$.

Now consider the case when the central character of
$\chi=\chi_{\lambda}$ is not
integral. Using the action of the Weyl group, one can choose
$ \lambda=\gamma\left(t\varepsilon_{0}\right)+\tau
$, where $ \tau $ is some integral dominant weight and $
t\notin{\mathbb Z} $. One can easily see that $
\eta+\rho:=\gamma\left(t\varepsilon_{0}\right)+\rho $ and $
\lambda+\rho $ are both regular and belong to the same Weyl chamber.
Therefore $ T_{\eta}^{\lambda} $ defines an equivalence of the
categories $ {\mathcal C}^{\eta} $ and $ {\mathcal C}^{\lambda} $
and hence the lemma holds for a non-integral central character.

Finally, let us consider the case when $\chi=\chi_{\lambda}$ is
singular integral. The conditions on $\lambda$ in Theorem
\ref{th1} ensure that one can choose $\lambda$ satisfying
$\left(\lambda+\rho,\varepsilon_{0}-\varepsilon_{k}\right)=0 $ for
exactly one $ 0<k\leq n $. Let us put $
\eta=\gamma\left(-k\varepsilon_{0}\right) $, $ \tau=\lambda-\eta
$. Then the stabilizers of $ \eta+\rho $ and $ \lambda+\rho $
coincide, $ \eta+\rho $ and $ \lambda+\rho $ belong to the same
Weyl chamber and $ \tau $ is a regular integral weight. Hence
$\tau$ is in the $W$-orbit of the highest weight of some
finite-dimensional $ {\mathfrak g} $-module $ V $. Therefore $
T_{\eta}^{\lambda} $ defines an equivalence of the categories $
{\mathcal C}^{\eta} $ and $ {\mathcal C}^{\lambda} $, and the
lemma holds in this case as well.\end{proof}

Denote by $\sigma$ the antiautomorphism of $\mathfrak g$ defined
by $\sigma (X)=X^t$. For any weight module  $M=\bigoplus_{\mu \in
\mathfrak{h}^*} M^{\mu}$ one can construct a new module
$M^{\vee}:=\bigoplus_{\mu \in \mathfrak{h}^*} (M^\mu)^*$ with
$\mathfrak g$-action defined by the formula $\langle Xu,m
\rangle=\langle u,\sigma(X)m \rangle$ for any $X\in\mathfrak g,
u\in (M^{\mu})^*,$ and $m\in M^{\nu}$. Then \;
${}^{\vee}:{\mathcal C}\to {\mathcal C}$ is a contravariant exact
functor, which maps $ {\mathcal C}^{\chi}_{\bar{\nu}} $ to itself.
If $\chi$ is singular or non-integral  $ {\mathcal
  C}^{\chi}_{\bar{\nu}}$ has only one simple object and therefore
$L^{\vee} \cong L$ for any simple module $L$ in $ {\mathcal
  C}^{\chi}_{\bar{\nu}}$. We show in Section 6 (Lemma \ref{lm32})
that every simple
module $L$ in $ {\mathcal  C}^{0}_{\bar{\nu}}$ can be obtained as
a unique simple submodule
and a unique simple quotient in
$T^\lambda_\eta (L')$ for some singular $\eta$ and simple module $L'$
in $ {\mathcal  C}^{\eta}$. Since ${}^\vee$ commutes with
$T^\lambda_\eta$ we obtain that
$L^{\vee} \cong L$ for any simple module $L$ in $\mathcal C$.

\section{Extensions between cuspidal modules }

Let $ {\mathfrak s} $ be a Lie subalgebra of $\mathfrak g$
containing $ {\mathfrak h}$. We consider the  functors Ext in the
category of $ {\mathfrak s} $-modules that are semisimple over $
{\mathfrak h} $. If $ M $ and $ N $ are two $ {\mathfrak s}
$-modules that are semisimple over $ {\mathfrak h} $, then  $
\operatorname{Ext}_{{\mathfrak s},\mathfrak h}^{i}\left(M,N\right) $ can be
expressed in terms of relative Lie algebra cohomology. In
particular,
\begin{equation}
\operatorname{Ext}_{\mathfrak s,\mathfrak h}^{1}\left(M,N\right)\cong
H^{1}\left({\mathfrak s},{\mathfrak
h};\operatorname{Hom}_{\mathbb C}\left(M,N\right)\right), \notag\end{equation}
where the right hand side is the corresponding relative cohomology
group (see \cite{Fuk} sections 1.3 and 1.4 for instance). For a sake of completeness we
recall the definition of $ H^{1}\left({\mathfrak s},{\mathfrak
h};\operatorname{Hom}_{\mathbb C}\left(M,N\right)\right) $. The set of $ 1
$-cocycles $ C^{1}\left({\mathfrak s},{\mathfrak
h};\operatorname{Hom}_{\mathbb C}\left(M,N\right)\right) $ is the subspace of
all $ c\in\operatorname{Hom}_{{\mathfrak h}}\left({\mathfrak
s},\operatorname{Hom}_{\mathbb C}\left(M,N\right)\right) $ such that
\begin{equation}
c\left({\mathfrak h}\right)=0\text{,
}c\left(\left[g_{1},g_{2}\right]\right)=\left[g_{1},c\left(g_{2}\right)\right]-\left[g_{2},c\left(g_{1}\right)\right]
\label{equ3}\end{equation}\myLabel{equ3,}\relax for any $
g_{1},g_{2}\in{\mathfrak s} $. A $ 1 $-cocycle $ c $ is a
coboundary if $ c\left(g\right)=\left[g,\varphi\right] $ for some
$ \varphi\in\operatorname{Hom}_{{\mathfrak h}}\left(M,N\right) $.
Denote by $ B^{1}\left({\mathfrak s},{\mathfrak
h};\operatorname{Hom}_{\mathbb C}\left(M,N\right)\right) $ the space of all
coboundaries. Then
\begin{equation}
H^{1}\left({\mathfrak s},{\mathfrak
h};\operatorname{Hom}_{\mathbb C}\left(M,N\right)\right):=C^{1}\left({\mathfrak
s},{\mathfrak
h};\operatorname{Hom}_{\mathbb C}\left(M,N\right)\right)/B^{1}\left({\mathfrak
s},{\mathfrak h};\operatorname{Hom}_{\mathbb C}\left(M,N\right)\right).
\notag\end{equation}

For any multiplicative subset $X$ of $U(\mathfrak{s})$
we denote the localization of
$U(\mathfrak{s})$ relative to $X$  by
$U_X(\mathfrak{s})$. For any $\mathfrak s$-module
$M$, ${\mathcal D}_X M := U_X(\mathfrak{s})\otimes_{U(\mathfrak s)} M$ denotes
the localization of $M$ relative to $X$.

\begin{lemma} \label{lm4}\myLabel{lm4}\relax  Let ${\mathfrak s}={\mathfrak h}\oplus{\mathfrak g}_{1} $ or
 ${\mathfrak s}={\mathfrak h}\oplus{\mathfrak g}_{- 1} $, $ M $ be a simple cuspidal $ {\mathfrak g} $-module,
and $ \mu\in\operatorname{supp} M $. Then $
\operatorname{End}_{{\mathfrak s}}\left(M\right)\cong
\operatorname{End}_{\mathbb C} M^{\mu} $ and
$ \operatorname{Ext}_{{\mathfrak
s},\mathfrak h}^{1}\left(M,M\right)=0 $.

\end{lemma}

\begin{proof} Let $\mathfrak s = {\mathfrak h}\oplus{\mathfrak g}_{1}$ (the case $\mathfrak s = {\mathfrak h}\oplus{\mathfrak g}_{-1}$ is treated
in the same way).  Let $ X_{1},\dots ,X_{n} $ be an $
\operatorname{ad}_{{\mathfrak h}} $-eigenbasis of $ {\mathfrak
g}_{1} $. Since $M$ is cuspidal, the action of $X_1,...,X_n$ is invertible.
Therefore the localization ${\mathcal D}_X M$ of $ M $
relative to $X:=\langle X_{1},\dots ,X_{n}\rangle \subset
U(\mathfrak{s}) $ is isomorphic to $ M $. In other words, $ M $ is
a module over $ U_X\left({\mathfrak s}\right) $. Moreover, if $\mu$ is a
weight of $M$, then $M$ is generated by $M^{\mu}$ as
$U_X(\mathfrak{s})$-module. Hence
$ M $ is isomorphic to the induced module $U_X(\mathfrak s)\otimes_{U(\mathfrak{h})} M^{\mu} $.   Therefore
\begin{equation}
\operatorname{End}_{{\mathfrak s}}\left(M\right) \cong
\operatorname{End}_{U_X\left({\mathfrak s}\right)}\left(M\right)
\cong \operatorname{Hom}_{U(\mathfrak h)}(M^{\mu},M)
\cong\operatorname{End}_{{\mathbb C}}\left(M^{\mu}\right).
\notag\end{equation}
Thus, we have an isomorphism $
\operatorname{End}_{{\mathfrak s}}\left(M\right)\cong
\operatorname{End}_{\mathbb C} M^{\mu} $. To prove the second statement note
that any $ {\mathfrak s} $-module $ M' $ (semisimple over $
{\mathfrak h} $ ) which can be included in an exact sequence of $
{\mathfrak s} $-modules
\begin{equation}
0 \to M \to M' \to M \to 0 \notag\end{equation} is a module over $
U_X\left({\mathfrak s}\right) $ since all $ X_{i} $ are
invertible. Since
$$U_X(\mathfrak s)\otimes_{U(\mathfrak{h})} M^{\mu}\simeq
{\mathbb C}\left[X_{1}^{\pm1},\dots
  ,X_{n}^{\pm1}\right]\otimes_{\mathbb{C}} M^{\mu}, $$
$ M $ is free over $ {\mathbb
C}\left[X_{1}^{\pm1},\dots ,X_{n}^{\pm1}\right] $, the exact
sequence splits over $ U_X\left({\mathfrak s}\right) $, and
therefore over $ {\mathfrak s} $ as well.\end{proof}

The following lemma is used in the next section.

\begin{lemma} \label{lm5}\myLabel{lm5}\relax  Let $ M $ be a simple cuspidal $ {\mathfrak g} $-module and $ c $ be a $ 1 $-cocycle
in $ C^{1}\left({\mathfrak g},{\mathfrak
h};\operatorname{End}\left(M\right)\right) $. Then there exists $
\psi \in\operatorname{End}_{{\mathfrak h}}\left(M\right) $ such
that for any $ g_{1}\in{\mathfrak g}_{1} $ and $
g_{0}\in{\mathfrak g}_{0} $,
\begin{equation}
c\left(g_{1}\right)=\left[g_{1},\psi\right]\text{, }\left[g_{0},\psi\right]+c\left(g_{0}\right)\in\operatorname{End}_{{\mathfrak g}_{1}}\left(M\right).
\notag\end{equation}
\end{lemma}

\begin{proof} The first identity follows directly  from the second
statement of Lemma~\ref{lm4} applied to $ {\mathfrak s}={\mathfrak
h}\oplus{\mathfrak g}_{1} $, since the restriction of $ c $ on $
{\mathfrak s} $ is a coboundary. To obtain the second statement,
use the identity
\begin{equation}
c\left(\left[g_{0},g_{1}\right]\right)=\left[g_{0},c\left(g_{1}\right)\right]-\left[g_{1},c\left(g_{0}\right)\right].
\notag\end{equation} Then
\begin{equation}
\left[\left[g_{0},g_{1}\right],\psi\right]=\left[g_{0},\left[g_{1},\psi\right]\right]-\left[g_{1},c\left(g_{0}\right)\right]
\notag\end{equation}
implies
\begin{equation}
\left[g_{1},\left[g_{0},\psi\right]+c\left(g_{0}\right)\right]=0
\notag\end{equation}
for any $ g_{1}\in{\mathfrak g}_{1} $. Hence $ \left[g_{0},\psi\right]+c\left(g_{0}\right)\in\operatorname{End}_{{\mathfrak g}_{1}}\left(M\right) $.\end{proof}

\begin{example} \label{ex1}\myLabel{ex1}\relax  Let $ {\mathfrak g}=\mathfrak{sl}\left(2\right) $ and $ M $ and $ N $ be two simple cuspidal $ {\mathfrak g} $-modules.
If $ \operatorname{Ext}_{{\mathfrak g}}^{1}\left(M,N\right)\not=0
$, then $ M $ and $ N $ are in the same block of $\mathcal C$.
Therefore by Corollary~\ref{cor2} and Theorem~\ref{th2}, $ M \cong
N $. Moreover, $ M $ is isomorphic to $ {\mathcal F}_{\mu} $ for
some $ \mu\in{\mathbb C}^{2} $. In particular, the weight
multiplicities of $ M $ are equal to 1. Let $ \{ X,H,Y \}$ be the
standard $ \mathfrak{sl}\left(2\right)$-basis, and $
c\in\operatorname{Hom}_{{\mathfrak h}}\left({\mathfrak
g},\operatorname{End}\left(M\right)\right) $ be a $ 1 $-cocycle.
Then $ c\left(H\right)=0 $, and, by Lemma~\ref{lm4} one may assume
without loss of generality that $ c\left(X\right)=0$, since one
can add a coboundary $d (m)$ such that $c(X)=[X,m]$. Then
$\left[X,c\left(Y\right)\right]=0 $ and therefore $c(Y)\in
 \operatorname{End}_{\mathfrak{s}}(M)$, where $\mathfrak{s}= {\mathbb C}H \oplus {\mathbb C} X$. Then, again by
 Lemma~\ref{lm4},
 $c\left(Y\right)=bX^{-1} $ for some $ b\in{\mathbb C} $. It is
straightforward to check that $ c\left(H\right)=c\left(X\right)=0
$ and $ c\left(Y\right)=bX^{-1} $ imply~\eqref{equ3}. Now let us
check that $ c $ is not trivial if $ b\not=0 $. Indeed, assume the
contrary and let $ c\left(g\right)=\left[g,\varphi\right] $ for
some $ \varphi\in\operatorname{End}_{{\mathfrak h}}\left(M\right)
$. But then $ \left[X,\varphi\right]=0 $, and again by
Lemma~\ref{lm4}, $ \varphi $ is a scalar map. Hence $ c=0 $.

Thus, $ \operatorname{Ext}_{{\mathfrak g},\mathfrak h}^{1}\left(M,N\right)=0 $
if $ M $ and $ N $ are not isomorphic, and $
\operatorname{Ext}_{{\mathfrak g},\mathfrak h}^{1}\left(M,N\right)={\mathbb C}
$ if $ M $ and $ N $ are isomorphic.

\end{example}

\section{The case of singular or non-integral central character }

In this section we compute $ \operatorname{Ext}_{\mathfrak g,\mathfrak
  h}^{1}\left(M,N\right)$  for $ n\geq 2 $ and irreducible
cuspidal modules $M$ and $N$ admitting a singular or a
non-integral central character $ \chi $. If $
\operatorname{Ext}_{{\mathfrak g,\mathfrak h}}^{1}\left(M,N\right)\not=0 $
then $M$ and $N$ must belong to the same block of $\mathcal C$,
therefore by Corollary~\ref{cor2}, $ M $ is isomorphic to $ N $.
Lemma~\ref{lm2} implies that it suffices to calculate
$\operatorname{Ext}^{1}_{\mathfrak g,\mathfrak h}\left(M,M\right) $ for the case $ M =
{\mathcal F}_{\mu} $ since any block $ {\mathcal
C}^{\lambda}_{\bar{\nu}}$ is equivalent to the one with
irreducible object ${\mathcal F}_{\mu} $. The main result of this
section is that for any cuspidal module ${\mathcal F}_{\mu}$,
$H^{1}({\mathfrak s},{\mathfrak h};\operatorname{End}_{\mathbb
C}\left({\mathcal F}_{\mu})\right)=\mathbb C$. Note that, the latter
  cohomology group describes  the space of infinitesimal
  deformations of ${\mathcal F}_{\mu}$ in $\mathcal C$ with the same
  support. On the other hand, the family
${\mathcal F}_{\mu+s(\varepsilon_0+\dots+\varepsilon_n)}$
provides a one-parameter deformation with desired properties,
hence the dimension of
$H^{1}({\mathfrak s},{\mathfrak h};\operatorname{End}_{\mathbb
C}\left({\mathcal F}_{\mu})\right)$ is at least one.
The difficult part is to show that that the dimension is not bigger.

\begin{lemma} \label{lm9}\myLabel{lm9}\relax  Let
\begin{equation}
z=t_{1}\frac{\partial}{\partial t_{1}}+\dots
+t_{n}\frac{\partial}{\partial
t_{n}}-nt_{0}\frac{\partial}{\partial t_{0}}, \notag\end{equation}
and in particular, $ {\mathfrak g}_{0}\cong
\mathfrak{sl}\left(n\right)\oplus{\mathbb C}z $. Then
\begin{equation}
{\mathcal F}_{\mu}=\bigoplus_{k \in {\mathbb Z}}{\mathcal
F}_{\mu}^{k}, \label{equ5}\end{equation}\myLabel{equ5,}\relax
where $ {\mathcal F}_{\mu}^{k} $ is the $ z $-eigenspace
corresponding to the eigenvalue $ |\mu|+\left(n+1\right)(k-\mu_{0})
$. Moreover, each $ {\mathcal F}_{\mu}^{k} $ is a simple cuspidal
$ \mathfrak{sl}\left(n\right) $-module isomorphic to $ {\mathcal
F}_{\left(\mu_{1}+k,\mu_{2},\dots ,\mu_{n}\right)} $.

\end{lemma}

\begin{proof} A straightforward calculation shows that
\begin{equation}
{\mathcal F}_{\mu}^{k}=t_{0}^{\mu_{0}-k}t_{1}^{\mu_{1}+k}\dots
t_{n}^{\mu_{n}}{\mathbb
C}\left[\left(\frac{t_{2}}{t_{1}}\right)^{\pm1},\dots
,\left(\frac{t_{n}}{t_{1}}\right)^{\pm1}\right].
\notag\end{equation} The $ \mathfrak{sl}\left(n\right) $-module
isomorphism ${\mathcal F}_{\mu}^{k}\cong {\mathcal
F}_{\left(\mu_{1}+k,\mu_{2},\dots ,\mu_{n}\right)} $ follows
directly from ~\eqref{equ1}. {}\end{proof}

\begin{lemma} \label{lm6}\myLabel{lm6}\relax  For $ n\geq2 $ and
$ u=\log \left(t_{0}t_{1}\dots t_{n}\right) $ let $ N:={\mathcal
F}_{\mu}\oplus u{\mathcal F}_{\mu} $ be the $\mathfrak g$-module
with action induced by the correspondence ~\eqref{equ1}. Then $ N
$ is a non-trivial self-extension of ${\mathcal F}_{\mu} $. The
cocycle defining this extension is given by the formulae $
c\left(E_{ij}\right)=\frac{t_{i}}{t_{j}} $, $ 0\leq i \neq j\leq n
$.

\end{lemma}

\begin{proof} It is obvious that $ N $ contains a submodule $ {\mathcal F}_{\mu} $
and $ N/{\mathcal F}_{\mu}\cong {\mathcal F}_{\mu} $. It remains
to check that this extension does not split. We will prove the
statement by induction on $n$.

Let us
start with $ {\mathfrak g}=\mathfrak{sl}\left(2\right) $ and show
that this self-extension is non-trivial for almost all $ \mu $.
Indeed, it is sufficient to show that the Casimir operator $
\Omega $ does not act as a scalar on $ M $. The Casimir operator
of $ \mathfrak{sl}\left(2\right) $ can be written in the following
form
\begin{equation}
\Omega=E_{01}E_{10}+E_{10}E_{01}+\frac{\left(E_{00}-E_{11}\right)^{2}}{2}.
\notag\end{equation} Then for $ f\in{\mathcal F}_{\mu} $,
\begin{equation}
\Omega\left(uf\right)=u\Omega\left(f\right)+\left(E_{01}\frac{t_{1}}{t_{0}}+\frac{t_{0}}{t_{1}}E_{10}+E_{10}\frac{t_{0}}{t_{1}}+\frac{t_{1}}{t_{0}}E_{01}\right)f.
\notag\end{equation}
But
\begin{equation}
E_{01}\frac{t_{1}}{t_{0}}+\frac{t_{0}}{t_{1}}E_{10}+E_{10}\frac{t_{0}}{t_{1}}+\frac{t_{1}}{t_{0}}E_{01}=2+2E=2(1+|\mu|).
\notag\end{equation}
Hence this self-extension is non-trivial for $ |\mu|\not=-1 $.

We now apply induction on $n$. Since ~\eqref{equ5} is the  $
{\mathfrak g}_{0} $-decomposition of $ {\mathcal F}_{\mu} $, the
restriction of $ c $ on
$\mathfrak{sl}\left(n\right)\subset{\mathfrak g}_{0}\subset
\mathfrak{sl} \left(n+1\right) $ is non-trivial because it is not
trivial on the component $ {\mathcal F}_{\mu}^{k} $ for almost all
$ k $. Hence $ c $ is not a trivial cocycle.\end{proof}

\begin{remark} \label{generalized}
One can easily generalize the construction in Lemma \ref{lm6} and obtain a family of non-trivial self-extensions of ${\mathcal F}_{\mu} $
in the category of generalized weight modules. Indeed, for
$(u_0,...,u_n) \in {\mathbb C}^{n+1}$ we define $
N(u_0,...,u_n):={\mathcal F}_{\mu}\oplus u{\mathcal F}_{\mu} $ for $u
= \sum_{i=0}^n u_i \log t_i$. Then it is easy to check that $
N(u_0,...,u_n)$ is a non-trivial self extension of ${\mathcal F}_{\mu} $ and is a generalized weight module but not a weight module unless $u_0=...=u_n$.

\end{remark}

\begin{lemma} \label{lm12}\myLabel{lm12}\relax  Let $V_0$ be a simple
  finite-dimensional $\mathfrak g_0$-module, and $n\geq 2$. A simple
cuspidal module $ {\mathcal F}_{\mu}\left(V_{0}\right) $ has a non-trivial
self-extension defined by the cocycle
$ c\left(E_{ij}\right)=\frac{t_{i}}{t_{j}} $ for all
$ i\not=j $, $ 0\leq i,j\leq n $.

\end{lemma}

\begin{proof} By using Lemma~\ref{lm9} and ~\ref{restr}
we obtain the following decomposition of $ {\mathcal
F}_{\mu}\left(V_{0}\right) $ as a $ {\mathfrak g}_{0} $-module
\begin{equation}
{\mathcal F}_{\mu}\left(V_{0}\right)=\bigoplus_{k\in{\mathbb
Z}}\left( {\mathcal F}_{\mu}^{k}\otimes V_{0}\right).
\notag\end{equation} As in the proof of Lemma~\ref{lm6} it
suffices to check that the following sequence of $\mathfrak
g_0$-modules
\begin{equation}
0 \to {\mathcal F}_{\mu}^{k}\otimes V_{0} \to \left({\mathcal
F}_{\mu}^{k}\oplus u{\mathcal F}_{\mu}^{k}\right)\otimes V_{0} \to
{\mathcal F}_{\mu}^{k}\otimes V_{0} \to 0 \notag\end{equation}
does not split. That is a consequence of the following general fact.
\end{proof}

\begin{lemma}\label{aux}Let $M$, $N$ and $L$ be modules over a Lie
  algebra $\mathfrak a$. If
\begin{equation}
0 \to M \to N \to L \to 0
\label{equ9}\end{equation}\myLabel{equ9,}\relax does not split,
then for any finite-dimensional $\mathfrak a$-module $ V $ the
sequence
\begin{equation}
0 \to M\otimes V \to N\otimes V \to L\otimes V \to 0
\notag\end{equation} does not split either.
\end{lemma}

\begin{proof}
Assume that the latter sequence splits. Then
\begin{equation}
0 \to M\otimes V\otimes V^{*} \to N\otimes V\otimes V^{*} \to
L\otimes V\otimes V^{*} \to 0 \notag\end{equation} also splits.
Suppose that $ p:L\otimes V\otimes V^{*} \to N\otimes V\otimes
V^{*} $ is a splitting map. Denote by $ i $ the natural embedding
$ L \to L\otimes V\otimes V^{*} $ and by $ j $ the natural
projection $ N\otimes V\otimes V^{*} \to N $. Then $ j\circ p\circ
i $ is a splitting map for~\eqref{equ9}.\end{proof}

\begin{remark}\label{rem44} Lemma \ref{lm12} still holds for
$n=1$ and $|\mu| \neq -1$ (see the proof of Lemma \ref{lm6}). In
the special case of $n=1$ and $|\mu| = -1$ one can easily check
that the cocycle $c$ is trivial. Indeed,  $c(g) = [g, \varphi]$,
for $\varphi \in \operatorname{End}_{\mathfrak h}(M)$ defined by
$\varphi(t^{\mu}) = \varphi(\mu)t^{\mu}$, where the function
$\varphi(\mu)$ can be found inductively using
 $\varphi(\mu_0, \mu_1) - \varphi(\mu_0 + 1, \mu_1 - 1) =
\frac{1}{\mu_1}$. Nevertheless, in this special case, we still
have a non-trivial cocycle (see Example \ref{ex1}).
\end{remark}

\begin{lemma} \label{lm8}\myLabel{lm8}\relax  Let
$ c\in\operatorname{Hom}_{{\mathfrak h}}\left({\mathfrak
g},\operatorname{End}_{\mathbb C}\left({\mathcal F}_{\mu}\right)\right) $ be
defined by the formulae $c(E_{ii})=0$, $
c\left(E_{ij}\right)=b_{ij}\frac{t_{i}}{t_{j}} $ for some $
b_{ij}\in{\mathbb C} $ if $i\neq j$. Then $ c $ is a $ 1 $-cocycle iff there exists
$ b\in{\mathbb C} $ such that $b_{ij}=b $ for all $ i\not=j $.

\end{lemma}

\begin{proof} Let $ b=b_{0 n} $ and $ c'\left(E_{ij}\right)=c\left(E_{ij}\right)-b\frac{t_{i}}{t_{j}} $. Then $ c' $ is a
cocycle. On the other hand, for any $k\neq 0$ and $k \neq n$,
\begin{equation}
c'\left(E_{kn}\right)=c'\left(\left[E_{k0},E_{0n}\right]\right)=-\left[E_{0n},\left(b_{k0}-b\right)\frac{t_{k}}{t_{0}}\right]=0
\notag\end{equation}
and for any $k\neq n$
\begin{equation}
0=c'\left(\left[E_{kn},E_{nk}\right]\right)=\left[E_{kn},c'\left(E_{nk}\right)\right]=\left[E_{kn},\left(b_{nk}-b\right)\frac{t_{n}}{t_{k}}\right]=b_{nk}-b.
\notag\end{equation} Hence $
c'\left(E_{nk}\right)=c'\left(E_{jn}\right)=0 $ for all $
j,k\not=n $. Then
\begin{equation}
c'\left(E_{jk}\right)=c'\left(\left[E_{jn},E_{nk}\right]\right)=0.
\notag\end{equation}
\end{proof}

\begin{lemma} \label{lm10}\myLabel{lm10}\relax  Let $ n\geq2 $, $ {\mathfrak p}={\mathfrak g}_{0}\oplus{\mathfrak g}_{1} $, and $ {\mathcal F}_{\mu} $ be cuspidal. The restriction map
\begin{equation}
r:H^{1}\left({\mathfrak g},{\mathfrak h};\operatorname{End}_{\mathbb
    C}\left({\mathcal F}_{\mu}\right)\right) \to H^{1}\left({\mathfrak
    p},{\mathfrak h};\operatorname{End}_{\mathbb C}\left({\mathcal F}_{\mu}\right)\right)
\notag\end{equation}
is injective.

\end{lemma}

\begin{proof} Let $ c\in C^{1}\left({\mathfrak g},{\mathfrak
      h};\operatorname{End}_{\mathbb C}\left({\mathcal F}_{\mu}\right)\right) $ be
  such that $ c\left({\mathfrak p}\right)$ is a coboundary.
 Without loss of generality we can choose $c$ so that
$ c\left({\mathfrak p}\right)=0$.
Then
\begin{equation}
[\mathfrak g_1,c(\mathfrak g_{-1})]\subset[\mathfrak
g_{-1},c(\mathfrak g_1)]+c(\mathfrak g_0)=0
\notag\end{equation}
implies
\begin{equation}
\left[{\mathfrak g}_{1},c\left(E_{0k}\right)\right]=0
\notag\end{equation} for all $k$. Since $E_{k0}\in \mathfrak g_1$,
$E_{k0}c(E_{0k})$ commutes with the action of $\mathfrak g_{1}$.
Moreover it commutes with the action $\mathfrak h$ since it maps
every weight space to itself. Therefore $E_{k0}c(E_{0k})\in
\operatorname {End}_{\mathfrak s}(\mathcal F_{\mu})$ for
$\mathfrak s=\mathfrak h \oplus {\mathfrak g}_1$, and the first
statement of Lemma~\ref{lm4} implies
\begin{equation}
c\left(E_{0k}\right)=b_{k}E_{k0}^{-1} \notag\end{equation} for
some constant $ b_{k}\in{\mathbb C} $. Choose $ i\geq1 $ and $
i\not=k $. Then
\begin{equation}
c\left(E_{0i}\right)=-c\left(\left[E_{ki},E_{0k}\right]\right)=-\left[E_{ki},b_{k}E_{k0}^{-1}\right]=0.
\notag\end{equation} That proves $ c\left(E_{0i}\right)=0 $ for
all $ 1\leq i\leq n $. Hence $ c=0 $.\end{proof}

Let $ F\left(\mu\right) $ be the set of functions $ \varphi:\mu+Q \to {\mathbb C} $. Then one can identify
$ F\left(\mu\right) $ with the space $ \operatorname{End}_{{\mathfrak h}}\left({\mathcal F}_{\mu}\right) $ by the formula
\begin{equation}
\varphi\left(t^{\lambda}\right)=\varphi\left(\lambda\right)t^{\lambda}.
\notag\end{equation} A function $\varphi(\lambda)$ that depends
only on its $i$-th coordinate $\lambda_i$ will be often written as
$\varphi(\lambda_i)$.

\begin{lemma} \label{lm11}\myLabel{lm11}\relax  Let $ {\mathcal
    F}_{\mu} $ be cuspidal and $ c\in\operatorname{Hom}_{{\mathfrak
      g}_{0}}\left({\mathfrak g}_{1},\operatorname{End}_{\mathbb C}\left({\mathcal F}_{\mu}\right)\right) $. Then
\begin{equation}
c\left(E_{i0}\right)=E_{i0}\phi \notag\end{equation} for some $
\phi\in F\left(\mu\right) $ such that $
\phi\left(\lambda\right)=\phi\left(\lambda_{0}\right) $ (i.e. $
\phi $ depends only on the first coordinate $ \lambda_{0} $ of $
\lambda\in\mu+Q $). Moreover, there exists some $
\zeta\left(\lambda\right)=\zeta\left(\lambda_{0}\right) $ such
that $ c\left(E_{i0}\right)=\left[E_{i0},\zeta\right] $.

\end{lemma}

\begin{proof} Lemma~\ref{lm9} implies
\begin{equation}
\operatorname{Hom}_{{\mathfrak g}_{0}}\left({\mathfrak
g}_{1},\operatorname{End}_{\mathbb C}\left({\mathcal F}_{\mu}\right)\right)
\cong \bigoplus_{k,l\in{\mathbb Z}}\operatorname{Hom}_{{\mathfrak
g}_{0}}\left({\mathfrak g}_{1}\otimes{\mathcal
F}_{\mu}^{k},{\mathcal F}_{\mu}^{l}\right). \notag\end{equation}
By comparing the eigenvalues of $ z $ one verifies that
\begin{equation}
\operatorname{Hom}_{{\mathfrak g}_{0}}\left({\mathfrak g}_{1}\otimes{\mathcal F}_{\mu}^{k},{\mathcal F}_{\mu}^{l}\right)=0
\notag\end{equation}
if $ l\not=k+1 $. We claim that
\begin{equation}
\operatorname{Hom}_{{\mathfrak g}_{0}}\left({\mathfrak
g}_{1}\otimes{\mathcal F}_{\mu}^{k},{\mathcal
F}_{\mu}^{k+1}\right)={\mathbb C}.
\label{equ99}\end{equation}\myLabel{equ99,}\relax This follows
from Remarks~\ref{rem99} and~\ref{rem100} applied to $ {\mathfrak
g}'=\mathfrak{sl}\left(n\right)\subset{\mathfrak g}_{0} $ and the
corresponding parabolic subalgebra $ {\mathfrak p}' $ of $
{\mathfrak g}' $. Denote by $V(\eta)$ the simple highest weight
${\mathfrak p}'$-module with highest weight $\eta$. Using the
isomorphism of $ {\mathfrak g}' $-modules
\begin{equation}
{\mathfrak g}_{1}\otimes{\mathcal F}_{\mu}^{k}\cong {\mathcal
F}_{\mu}^{k}\left({\mathfrak g}_{1}\right), \notag\end{equation}
and the exact sequence of $ {\mathfrak p}' $-modules
\begin{equation}
0 \to V({\varepsilon_{2}}) \to {\mathfrak g}_{1} \to
V({\varepsilon_{1}}) \to\text{ 0,} \notag\end{equation} we obtain
the following exact sequence
\begin{equation}
0 \to {\mathcal F}_{\mu}^{k}\left(V({\varepsilon_{2}})\right) \to
{\mathcal F}_{\mu}^{k}\otimes{\mathfrak g}_{1} \to {\mathcal
F}_{\mu}^{k}\left(V({\varepsilon_{1}})\right)={\mathcal
F}_{\mu}^{k+1} \to\text{ 0.} \notag\end{equation} Since $
\operatorname{Hom}_{{\mathfrak g}_{0}}\left({\mathcal
F}_{\mu}^{k}\left(V({\varepsilon_{2}})\right),{\mathcal
F}_{\mu}^{k+1}\right)=0 $, the last exact sequence
implies~\eqref{equ99}.

Obviously the map $E_{i0}\otimes v \to E_{i0}v$ defines a
$\mathfrak g_0$-module homomorphism ${\mathfrak
g}_{1}\otimes{\mathcal F}_{\mu}\to{\mathcal F}_{\mu}$ . Therefore
any non-zero homomorphism $ c\in \operatorname{Hom}_{{\mathfrak
g}_{0}}\left({\mathfrak g}_{1}\otimes{\mathcal F}_{\mu},{\mathcal
F}_{\mu}\right) $ can be written in the form
\begin{equation}
c\left(E_{i0}\otimes v\right)=\phi\left(\mu_{0}-k\right)E_{i0}v
\notag\end{equation} where $ v\in{\mathcal F}_{\mu}^{k} $. That
implies the first statement of the lemma.

To prove the second statement note that the equation $ c\left(E_{i0}\right)=\left[E_{i0},\zeta\right] $
is equivalent to the following functional equation
\begin{equation}
-\left(\lambda_{0}-1\right)\zeta\left(\lambda_{0}-1\right)+\lambda_{0}\zeta\left(\lambda_{0}\right)=\phi\left(\lambda_{0}\right).
\notag\end{equation}
Such $ \zeta $ can be easily found inductively since $ \lambda_{0} $ is never 0, as $ \mu_{0}\notin{\mathbb Z} $.\end{proof}

\begin{lemma} \label{lm7}\myLabel{lm7}\relax  If $ {\mathcal F}_{\mu}
  $ is cuspidal, then $ \operatorname{Ext}_{{\mathfrak g},\mathfrak h}^{1}\left({\mathcal F}_{\mu},{\mathcal F}_{\mu}\right)={\mathbb C} $.

\end{lemma}

\begin{proof} By Lemma~\ref{lm6} it suffices to prove that $ \dim
\operatorname{Ext}_{{\mathfrak g,\mathfrak h}}^{1}\left({\mathcal
F}_{\mu},{\mathcal F}_{\mu}\right)\leq1 $.

Let
$ c\in C^{1}\left({\mathfrak g},{\mathfrak h},
  \operatorname{End}_{\mathbb C}
\left({\mathcal F}_{\mu}\right)\right) $. As follows from
Lemma~\ref{lm5}, we may assume that there is $ \psi\in
F\left(\mu\right) $ such that for all $ g_{0}\in{\mathfrak
g}_{0},$ and $g_{1}\in{\mathfrak g}_{1} $,
\begin{equation}
c\left(g_{1}\right)=\left[g_{1},\psi\right]\text{,
}\left[g_{0},\psi\right]+c\left(g_{0}\right)\in\operatorname{End}_{{\mathfrak
g}_{1}}\left({\mathcal F}_{\mu}\right). \notag\end{equation}
Let $\mathfrak s=\mathfrak h\oplus {\mathfrak g}_1$. Note that
$\frac{t_{i}}{t_{j}}\in \operatorname{End}_{{\mathfrak
g}_{1}}\left({\mathcal F}_{\mu}\right)$ if $i,j>0$ and
$$\frac{t_j}{t_i}([E_{ij},\psi]+c(E_{ij}))\in\operatorname{End}_{{\mathfrak
s}}\left({\mathcal F}_{\mu}\right).$$ Therefore,  by the first
statement of Lemma~\ref{lm4}, for any $ i\not=j $, $ 1\leq i,j\leq
n $, there is a constant $ b_{ij}\in{\mathbb C} $ such that
\begin{equation}
\left[E_{ij},\psi\right]+c\left(E_{ij}\right)=b_{ij}\frac{t_{i}}{t_{j}}.
\notag\end{equation} But then $c'\left(E_{ij}\right):=c(E_{ij})+[E_{ij},\psi]=b_{ij}\frac{t_{i}}{t_{j}} $ is a cocycle on
$ {\mathfrak g}_{0} $. Lemmas ~\ref{lm9} and \ref{lm8} imply that
$ b=b_{ij} $ for some constant $ b $, and thus $ c_{|{\mathfrak
g}_{0}} $ is equivalent to the cocycle $ b\frac{t_{i}}{t_{j}} $
modulo some coboundary. Therefore, we may assume that
\begin{equation}
c\left(E_{ij}\right)=b\frac{t_{i}}{t_{j}}, \notag\end{equation}
for $ 1\leq i \neq j\leq n $. Then $
\left[c\left(E_{ij}\right),{\mathfrak g}_{1}\right]=0 $, and one has
\begin{equation}
\left[g_{0},c\left(g_{1}\right)\right]=c\left(\left[g_{0},g_{1}\right]\right)
\notag\end{equation} for all $ g_{0}\in{\mathfrak g}_{0}$ and
$g_{1}\in{\mathfrak g}_{1} $. By Lemma~\ref{lm11}, this implies
\begin{equation}
c\left(g_{1}\right)=\left[\zeta,g_{1}\right]
\notag\end{equation}
for some $ \zeta=\zeta\left(\lambda_{0}\right) $. But $ \left[{\mathfrak g}_{0},\zeta\right]=0 $. Therefore, the cocycle
\begin{equation}
c'\left(g\right)=c\left(g\right)+\left[g,\zeta\right]
\notag\end{equation}
defines the same cohomology class as $ c $ and
\begin{equation}
c'\left(E_{ij}\right)=b\frac{t_{i}}{t_{j}}\text{, }c'\left(E_{i0}\right)=0
\notag\end{equation}
for $ i\not=j $, $ 1\leq i,j\leq n $. Therefore, $ \dim
H^{1}\left({\mathfrak p},{\mathfrak h};\operatorname{End}_{\mathbb C}\left({\mathcal F}_{\mu}\right)\right)=1 $ and, by Lemma~\ref{lm10}, $
\dim  H^{1}\left({\mathfrak g},{\mathfrak
    h};\operatorname{End}_{\mathbb C}\left({\mathcal F}_{\mu}\right)\right)\leq1 $.\end{proof}

\begin{theorem} \label{th3}\myLabel{th3}\relax  Let $ M $ be a simple cuspidal $ {\mathfrak g} $-module with singular or
non-integral central character $ \chi $ and let $ N $ be any
simple cuspidal module. Then $ \operatorname{Ext}_{{\mathfrak
g},\mathfrak h}^{1}\left(M,N\right)=0 $ if $ N $ is not isomorphic to $ M $
and $ \operatorname{Ext}_{{\mathfrak
g},\mathfrak h}^{1}\left(M,M\right)={\mathbb C} $.

\end{theorem}

\begin{proof} The theorem follows from Lemmas ~\ref{lm2} and~\ref{lm7}
  since any block of ${\mathcal C}$ with non-integral or singular central character is
  isomorphic to a block with unique simple module $\mathcal F_{\mu}$
  for some $\mu$. {}\end{proof}

\begin{theorem}\label{genthm} If $ {\mathcal F}_{\mu}$ is cuspidal,
then $ \operatorname{Ext}_{{\mathfrak g}}^{1}\left({\mathcal F}_{\mu},{\mathcal F}_{\mu}\right)={\mathbb C}^{n+1}$.
\end{theorem}

\begin{proof} If $M$ is a self-extension of $ {\mathcal F}_{\mu}$,
  then $\mathfrak h$ acts locally finitely on $M$ and therefore $M$ is
  a generalized weight module. On the other hand, we have an
  isomorphism
$$\mbox{Ext}^1_{\mathfrak g} (A, B) \simeq H^1({\mathfrak g},
\mbox{Hom}_{\mathbb C} (A,B))$$
for any $\mathfrak g$-modules $A$ and $B$ (see \cite{Fuk}). Since in
our case the extension is a generalized weight module, we can assume
without loss of generality that the cocycle defining it is
$\mathfrak h$-invariant, i.e. $c\in \mbox{Hom}_{\mathfrak h}(\mathfrak
g,\mbox{End}_{\mathbb C}({\mathcal F}_{\mu}))$.

Use the same notations as in the proof of Lemma \ref{lm4} and let
${\mathfrak s}={\mathfrak h}\oplus{\mathfrak g}_{1} $. Since
${\mathcal F}_{\mu} $ is free over $\mathbb C[X_1^{\pm 1},\dots,
X_n^{\pm 1}]$ we may assume that $c(X_i)=0$ for all $i\leq n$. Since
$0=c([h,X_i])=[c(h),X_i]$ we obtain that $c(h)\in
\mbox{End}_{U_X(\mathfrak s)}({\mathcal F}_{\mu})$  for all $h\in
\mathfrak h$. But
$\mbox{End}_{U_X(\mathfrak s)}({\mathcal F}_{\mu})=\mathbb C$, so $c(h)$
is a constant for each $h\in \mathfrak h$. Let $u_i=c(E_{ii}-E_{00})$,
then, as we explained already in Remark \ref{generalized}, the linear
functional $c'(X)=X(u)$ where $u=\sum_{i=1}^n u_i \mbox{log} t_i$
defines a non-trivial cocycle $c'\in \mbox{Hom}_{\mathfrak h}(\mathfrak
g,\mbox{End}_{\mathbb C}({\mathcal F}_{\mu}))$. Moreover, $c'(X_i)=0$
and $c'(E_{ii}-E_{00})=u_i\in \mathbb C$. Let $c''=c-c'$.
Then $c''(h)=0$ for any $h\in \mathfrak h$,
and therefore $c''\in C^1(\mathfrak g,\mathfrak h,
\mbox{End}_{\mathbb C}({\mathcal F}_{\mu}))$. By Theorem \ref{th3} we
have   $H^1(\mathfrak g,\mathfrak h,
\mbox{End}_{\mathbb C}({\mathcal F}_{\mu}))=\mathbb C$, and hence
$H^1(\mathfrak g,
\mbox{End}_{\mathbb C}({\mathcal F}_{\mu}))=\mathbb C^{n+1}$. Theorem \ref{genthm}
is proven.
\end{proof}

\section{An extension $\bar {\mathcal C}$ of $\mathcal C$
and the structure of the category ${\protect \mathcal C}^{\chi}_{\bar{\nu}}$
 for non-integral or singular $ \chi $}

Let $\mathcal A$ be an abelian category and $P$ be a projective
generator in $\mathcal A$. It is a well-known fact (see, for example,
~\cite{GM} exercise 2, section 2.6) that the functor $\mbox{Hom}_{\mathcal A}(P,M)$ provides an
equivalence of $\mathcal A$ and the category of right modules over the
ring $\mbox{Hom}_{\mathcal A}(P,P)$. In case when every object in
$\mathcal A$ has a finite length and each simple object has a projective
cover, one reduces the problem of classifying indecomposable objects
in $\mathcal A$ to the similar problem for modules over a
finite-dimensional algebra (see \cite{Gab2},\cite{Gab3}). In many
cases when $\mathcal A$ does not have projective modules it is
possible to consider a certain completion of $\mathcal A$  and reduce the case to
the category of modules over some pointed algebra. We use this
strategy to study the category $\mathcal C$ of cuspidal
modules. However, we use injective modules instead of
projectives and exploit the existence of the duality functor
${}^\vee$ in  $\mathcal C$. We prefer this consideration since
in this case we avoid taking projective limits and introducing
topology. Another advantage of this approach is that the center of
$U(\mathfrak g)$ acts locally finitely on injective limits of cuspidal
modules.

Let $ \bar{\mathcal C}$ be the full subcategory of all weight modules
consisting of $\mathfrak g$-modules
$ M $ which have countable dimension and whose  finitely generated submodules
 belong to $ {\mathcal C}$. It is not hard to see
that every such $M$ has an exausting filtration
$0\subset M_1\subset M_2\subset\dots$ such that each $M_i\in\mathcal
C$. It implies that the action of the center $Z$ of the universal
enveloping algebra $U$ on $M$ is locally finite and we have a
decomposition
\begin{equation}
\bar{\mathcal C}=\bigoplus_{\chi \in Z',\; \bar{\nu} \in {\mathfrak
h}^*/Q}\bar{\mathcal C}^{\chi}_{\bar{\nu}}, \notag\end{equation}
defined in the same way as for $\mathcal C$.

Before we proceed with studying blocks of $\bar{\mathcal C}$
let us formulate a general result.
Let $R$ be a unital $\mathbb
C$-algebra and let $\bar{\mathcal
 A}$ be an abelian category of $R$-modules
  satisfying the following conditions:

$\bullet$ $\bar{\mathcal A}$ contains finitely many up to isomorphism simple
objects $L_1,...,L_n$ such that $\mbox{End}_R (L_i)=\mathbb C$;

$\bullet$ $\bar{\mathcal A}$ contains indecomposable injective
  modules $I_1,...,I_n$ such that $\operatorname {Hom}_{R}(L_i,I_j)=0$
  if $i\neq j$, and  $\operatorname {Hom}_{R}(L_i,I_i)=\mathbb C$.

$\bullet$ Let
  ${\mathcal A}$ be the subcategory of  $\bar{\mathcal A}$ which
  consists of all objects in $\bar{\mathcal A}$ of finite
  length. Assume that every module $M$ in $\bar{\mathcal A}$ has an
  increasing exausting filtration
\begin{equation}
0=F^0(M)\subset F^1(M)\subset \dots \subset F^k(M)\subset \dots
\end{equation}
 such that $F^k(M)\in \mathcal A$ for all
  $k$.

$\bullet$ Finally, assume that there exists an involutive
contravariant exact faithful functor
  \; ${}^{\vee}: \mathcal A \to \mathcal A$ such that
$$\operatorname{Hom}_{R}(M,N)
 \cong \operatorname{Hom}_{R}(N^{\vee},M^{\vee}).$$

Let $I:=I_1\oplus\dots\oplus I_n$ and $\mathcal E:=\operatorname
  {End}_R(I)$. Define a functor $\Phi$ from $\mathcal A$ to
  $\mathcal E$-mod by
\begin{equation}
\Phi(M):=\operatorname {Hom}_R(M^{\vee},I).
\end{equation}

\begin{theorem} \label{th5}\myLabel{th5}\relax The functor $\Phi$
  establishes an equivalence of the category  $\mathcal A$ and the category of all
  finite-dimensional $\mathcal E$-modules.
\end{theorem}

\begin{proof} The functor $\Phi$ is exact as follows from the injectivity of
  $I$. It is  straightforward that $\Phi(L_1),\dots,\Phi(L_n)$ are
  pairwise non-isomorphic one-dimensional $\mathcal E$-modules.
Therefore $\Phi$ maps a simple object to a
  one-dimensional $\mathcal E$-module, hence an object of finite length to
  a finite-dimensional $\mathcal E$-module.

Next we will show that if $V$ is a simple
finite-dimensional $\mathcal
  E$-module, then $V\cong \Phi(L_i)$. The conditions imposed on the category
$\bar{\mathcal A}$ ensure that $I$ has a filtration
$0=F^0(I)\subset F^1(I)\subset \dots \subset F^k(I)\subset \dots$
such
  that $F^1(I)$ is a maximal semisimple submodule in $I$, (in fact,
$F^1(I)\cong L_1\oplus\dots\oplus L_n$) and
  $F^k(I)/F^{k-1}(I)$ is semisimple for all $k>0$. Let
\begin{equation}
{\mathcal E}':=\{\phi \in {\mathcal E} | \phi(F^1(I))=0\}.
\end{equation}
It is easy to check that ${\mathcal E}'$ is a two sided ideal in
$\mathcal E$, and
\begin{equation}\label{qn1}
\mathcal E / \mathcal E' \cong \operatorname
{Hom}_R(L_1\oplus\dots\oplus L_n,I)\simeq\operatorname{End
  }_R(L_1\oplus\dots\oplus L_n).
\end{equation}
Moreover, any $\phi \in \mathcal E'$ is locally nilpotent, because
$\phi^k (F^k(I))=0$. Hence $c+\phi$ is invertible for any non-zero
$c\in \mathbb C$. Therefore the only eigenvalue of $\phi$ in $V$
is zero, and in particular, every $\phi\in {\mathcal E}'$ acts
nilpotently on $V$. That implies $({\mathcal E}')^N(V)=0$. By the
simplicity of $V$, ${\mathcal E}'(V)=0$. Now the statement follows
directly from ~(\ref{qn1}).

Now consider the natural isomorphism
$$\operatorname{Hom}_{R}(M,\operatorname{Hom}_{\mathcal
  E}(F,I))\cong \operatorname{Hom}_{\mathcal
  E}(F,\operatorname{Hom}_{R}(M,I))$$
for any $\mathcal E$-module $F$ and $M\in\bar{\mathcal A}$.
If  $F$ is a finite-dimensional $\mathcal
  E$-module, it is not dificult to see by induction on dim $F$ that
  $\operatorname{Hom}_{\mathcal E}(F,I)$ has a finite length as an
  $R$-module and hence lies in $\mathcal A$. Therefore for any
  $M\in\mathcal A$ we have
$$\operatorname{Hom}_{R}((\operatorname{Hom}_{\mathcal
  E}(F,I))^{\vee},M)\cong\operatorname{Hom}_{R}(M^{\vee},\operatorname{Hom}_{\mathcal
  E}(F,I))\cong \operatorname{Hom}_{\mathcal
  E}(F,\operatorname{Hom}_{R}(M^{\vee},I)).$$

Thus, the functor $\Psi$ from the category of finite-dimensional
$\mathcal E$-modules to ${\mathcal A}$ defined
by
$$\Psi(F)=(\operatorname{Hom}_{\mathcal E}(F,I))^{\vee}$$
is the right adjoint of $\Phi$.
It is obvious that $\Psi(\Phi(L_i))\simeq L_i$ for all $i$.
That implies Theorem \ref{th5}.
\end{proof}

Let $ \chi $ be a non-integral or singular central character and
$ \bar{\nu}\in{\mathfrak h}^{*}/Q $ be such that $ {\mathcal
C}^{\chi}_{\bar{\nu}} $ is not empty.
By Corollary~\ref{cor2}, there is exactly one up to
isomorphism simple object in $ {\mathcal C}^{\chi}_{\bar{\nu}} $,
which is isomorphic to $ {\mathcal F}_{\mu}\left(V_{0}\right) $
for suitable $ \mu $ and $ V_{0} $. Define the $\mathfrak
g$-modules (see Lemma ~\ref{lm12})
\begin{equation}
{\mathcal F}_{\mu}^{\left(m\right)}:={\mathcal F}_{\mu}\oplus
u{\mathcal F}_{\mu}\oplus\dots \oplus u^{m}{\mathcal
F}_{\mu}\text{, }{\mathcal
F}_{\mu}^{\left(m\right)}\left(V_{0}\right):={\mathcal
F}_{\mu}^{\left(m\right)}\otimes_{{\mathcal
O}}\Gamma\left(U,{\mathcal V}_{0}\right). \notag\end{equation} For
$n\geq 2$ or $|\mu|\neq -1$ the action of $\mathfrak g$ on
${\mathcal F}_{\mu}^{\left(m\right)}$ is the standard one. For
$n=1$ and $|\mu| = -1$ we set $X(u^m \otimes f):=u^m \otimes Xf$,
$H(u^m \otimes f):=u^m \otimes Hf$, and $Y(u^m \otimes f):=u^m
\otimes Yf + u^{m-1} \otimes X^{-1}f$, where $f \in {\mathcal
F}_{\mu}$. Note that the standard action of $\mathfrak g$ in the
latter case would lead to semisimple modules ${\mathcal
F}_{\mu}^{\left(m\right)}$ (see Remark \ref{rem44}). In the proofs
of the results in this section we assume that the action is
standard, i.e. $n\geq 2$ or $|\mu|\neq -1$. However, it is not hard
to check that all results remain valid in the exceptional case as
well. The details are left to the reader.

\begin{lemma} \label{lm24}\myLabel{lm24}\relax  $ {\mathcal F}_{\mu}^{\left(m\right)}\left(V_{0}\right) $ is an indecomposable module.

\end{lemma}

\begin{proof} First, note that $ {\mathcal F}_{\mu}^{\left(m\right)}\left(V_{0}\right) $ has a filtration
\begin{equation}
0\subset{\mathcal F}_{\mu}\left(V_{0}\right)\subset{\mathcal
F}_{\mu}^{\left(1\right)}\left(V_{0}\right)\subset\dots
\subset{\mathcal F}_{\mu}^{\left(m\right)}\left(V_{0}\right).
\notag\end{equation} To check that $ {\mathcal
F}_{\mu}^{\left(m\right)}\left(V_{0}\right) $ is indecomposable it
is sufficient to check that $ {\mathcal
F}_{\mu}^{\left(m\right)}\left(V_{0}\right) $ has a unique
irreducible submodule. We prove this by induction on $m$. Assume
that $ {\mathcal F}_{\mu}^{\left(m-1\right)}\left(V_{0}\right) $
has a unique irreducible submodule, and let $ {\mathcal
F}_{\mu}^{\left(m\right)}\left(V_{0}\right) $ have an irreducible
submodule $ L\not={\mathcal F}_{\mu}\left(V_{0}\right) $. Then $
L\cap{\mathcal F}_{\mu}^{\left(m-1\right)}\left(V_{0}\right)=0 $
and thus
\begin{equation}
{\mathcal F}_{\mu}^{\left(m\right)}\left(V_{0}\right)\cong
L\oplus{\mathcal F}_{\mu}^{\left(m-1\right)}\left(V_{0}\right).
\notag\end{equation} But in this case
\begin{equation}
{\mathcal F}_{\mu}^{\left(1\right)}\left(V_{0}\right)\cong
{\mathcal F}_{\mu}^{\left(m\right)}\left(V_{0}\right)/{\mathcal
F}_{\mu}^{\left(m-2\right)}\left(V_{0}\right)\cong {\mathcal
F}_{\mu}\left(V_{0}\right)\oplus L. \notag\end{equation} However,
by Lemma~\ref{lm12}, $ {\mathcal
F}_{\mu}^{\left(1\right)}\left(V_{0}\right) $ is indecomposable.
Contradiction.\end{proof}

\begin{lemma} \label{lm25}\myLabel{lm25}\relax  $
  \operatorname{Ext}^{1}_{\mathfrak g,\mathfrak h}\left({\mathcal F}_{\mu}\left(V_{0}\right),{\mathcal F}_{\mu}^{\left(m\right)}\left(V_{0}\right)\right)={\mathbb C} $.

\end{lemma}

\begin{proof} We again apply induction on $ m $. For $ m=0 $ the statement follows from
Theorem \ref{th3}. We use now the exact sequence
\begin{equation}
0 \to {\mathcal F}_{\mu}^{\left(m-1\right)}\left(V_{0}\right) \to
{\mathcal F}_{\mu}^{\left(m\right)}\left(V_{0}\right) \to
{\mathcal F}_{\mu}\left(V_{0}\right) \to\text{ 0.}
\notag\end{equation} Since $ \operatorname{Hom}_{\mathfrak g}\left({\mathcal
F}_{\mu}\left(V_{0}\right),{\mathcal
F}_{\mu}^{\left(k\right)}\left(V_{0}\right)\right)={\mathbb C} $
for all $ k $, and $ \operatorname{Ext}^{1}_{\mathfrak g,\mathfrak h}\left({\mathcal
F}_{\mu}\left(V_{0}\right),{\mathcal
F}_{\mu}^{\left(m-1\right)}\left(V_{0}\right)\right)={\mathbb C}
$, by the inductive assumption, the corresponding long exact sequence
of Ext starts with
\begin{equation}
0 \to {\mathbb C} \to {\mathbb C} \to {\mathbb C} \to {\mathbb C}
\to \operatorname{Ext}^{1}_{\mathfrak g,\mathfrak h}\left({\mathcal
F}_{\mu}\left(V_{0}\right),{\mathcal
F}_{\mu}^{\left(m\right)}\left(V_{0}\right)\right) \to {\mathbb C}
\to \dots \notag\end{equation} Therefore, $ \dim
\operatorname{Ext}^{1}_{\mathfrak g,\mathfrak h}\left({\mathcal
F}_{\mu}\left(V_{0}\right),{\mathcal
F}_{\mu}^{\left(m\right)}\left(V_{0}\right)\right)\leq1 $. On the
other hand, Lemma \ref{lm24} implies that $ {\mathcal
F}_{\mu}^{\left(m+1\right)}\left(V_{0}\right) $ is a non-trivial
extension of $ {\mathcal F}_{\mu}\left(V_{0}\right) $ by $
{\mathcal F}_{\mu}^{\left(m\right)}\left(V_{0}\right) $. Hence $
\operatorname{Ext}^{1}_{\mathfrak g,\mathfrak h}\left({\mathcal
F}_{\mu}\left(V_{0}\right),{\mathcal
F}_{\mu}^{\left(m\right)}\left(V_{0}\right)\right)={\mathbb C} $.
\end{proof}

A natural example of a module
in $ \bar{\mathcal C}^{\chi}_{\bar{\nu}} $  is
\begin{equation}
\bar{{\mathcal F}}_{\mu}\left(V_{0}\right):=\lim
_{\longrightarrow}{\mathcal
F}_{\mu}^{\left(m\right)}\left(V_{0}\right)=\bigoplus_{m\geq0}u^{m}{\mathcal
F}_{\mu}\left(V_{0}\right). \notag\end{equation}

\begin{lemma} \label{lm26}\myLabel{lm26}\relax  $ \bar{{\mathcal F}}_{\mu}\left(V_{0}\right) $ is an indecomposable
 injective object in  $ \bar{{\mathcal C}}^{\chi}_{\bar{\nu}}$,
 and\\
$\operatorname{End}_{{\mathfrak g}}\left(\bar{{\mathcal
F}}_{\mu}\left(V_{0}\right)\right)= {\mathbb
C}\left[\left[\frac{\partial}{\partial u}\right]\right] $.

\end{lemma}

\begin{proof} $ \bar{{\mathcal F}}_{\mu}\left(V_{0}\right) $ is indecomposable since it contains a unique
simple submodule. The latter follows from Lemma~\ref{lm24}. To
verify the endomorphism identity note that
$$\operatorname{End}_{{\mathfrak g}}\left(\bar{{\mathcal
F}}_{\mu}\left(V_{0}\right)\right) = \lim _{\longleftarrow}
\operatorname{End}_{{\mathfrak g}}\left({\mathcal
F}_{\mu}^{(k)}\left(V_{0}\right)\right)\mbox{ and
}\operatorname{End}_{{\mathfrak g}}\left({\mathcal
F}_{\mu}^{\left(k\right)}\left(V_{0}\right)\right)\cong {\mathbb
C}\left[\frac{\partial}{\partial
u}\right]/\left(\frac{\partial^{k}}{\partial u^{k}}\right). $$ To
prove the injectivity it suffices to show that $
\operatorname{Ext}^{1}_{\mathfrak g,\mathfrak h}\left({\mathcal
F}_{\mu}\left(V_{0}\right),\bar{{\mathcal
F}}_{\mu}\left(V_{0}\right)\right)=0 $. \footnote{Indeed, if
$
\operatorname{Ext}^{1}_{\mathfrak g,\mathfrak h}\left({\mathcal
F}_{\mu}\left(V_{0}\right),\bar{{\mathcal
F}}_{\mu}\left(V_{0}\right)\right)=0 $, then $
\operatorname{Ext}^{1}_{\mathfrak g,\mathfrak h}\left(M,\bar{{\mathcal
F}}_{\mu}\left(V_{0}\right)\right)=0 $ for any cuspidal module
$M$. Since any module in $ \bar{{\mathcal C}}^{\chi}_{\bar{\nu}}$ is
an injective limit of cuspidal ones,  $
\operatorname{Ext}^{1}_{\mathfrak g,\mathfrak h}\left(M,\bar{{\mathcal
F}}_{\mu}\left(V_{0}\right)\right)=0 $ for any module
$M$ in $ \bar{{\mathcal C}}^{\chi}_{\bar{\nu}}$.} Assume the opposite.
Let $
c\in\operatorname{Hom}_{{\mathfrak h}}\left({\mathfrak
g}\otimes{\mathcal F}_{\mu}\left(V_{0}\right),\bar{{\mathcal
F}}_{\mu}\left(V_{0}\right)\right) $ be a non-trivial cocycle that induces an
exact sequence
\begin{equation}
0 \to \bar{{\mathcal F}}_{\mu}\left(V_{0}\right) \to M \to
{\mathcal F}_{\mu}\left(V_{0}\right) \to\text{ 0.}
\notag\end{equation} Pick $ m\in M $ such that $
m\notin\bar{{\mathcal F}}_{\mu}\left(V_{0}\right) $ and let $
M':=U\left({\mathfrak g}\right)m $. Since $M'$ is finitely generated,
 $ M'\cap\bar{{\mathcal F}}_{\mu}\left(V_{0}\right)={\mathcal
F}_{\mu}^{\left(k\right)}\left(V_{0}\right) $ for some $ k $.
Since ${\mathcal F}_{\mu}\left(V_{0}\right)$ is simple, we have
the following exact sequence
\begin{equation}
0 \to {\mathcal F}_{\mu}^{\left(k\right)}\left(V_{0}\right)
\to M' \to {\mathcal F}_{\mu}\left(V_{0}\right)
\to\text{ 0.} \notag\end{equation}
If we identify $M$ with $\mathcal F_{\mu}(V_0)\oplus \bar{\mathcal
  F}_{\mu}(V_0)$ as a vector space, then the action of $g\in \mathfrak
g$ on $M$ is given by $g(m_1,m_2)=(g m_1,c(g) m_1+g m_2)$. Since
$M'=\mathcal F_{\mu}(V_0)\oplus\mathcal F_{\mu}^{(k)}(V_0)$ is
$\mathfrak g$-invariant,
$c\in\operatorname{Hom}_{{\mathfrak h}}\left({\mathfrak
g}\otimes{\mathcal F}_{\mu}\left(V_{0}\right),{\mathcal
F}_{\mu}^{\left(k\right)}\left(V_{0}\right)\right) $.
Now consider the exact sequence
\begin{equation}
0 \to {\mathcal F}_{\mu}^{\left(k\right)}\left(V_{0}\right) \to
\bar{{\mathcal F}}_{\mu}\left(V_{0}\right) \xrightarrow[]{\varphi}
\bar{{\mathcal F}}_{\mu}\left(V_{0}\right) \to\text{ 0,}
\notag\end{equation} where $ \varphi=\frac{\partial^{k}}{\partial
u^{k}} $. This sequence leads to the long exact sequence
\begin{equation}
0 \to \operatorname{Hom}_{{\mathfrak g}}\left({\mathcal F}_{\mu}\left(V_{0}\right),{\mathcal F}_{\mu}^{\left(k\right)}\left(V_{0}\right)\right)={\mathbb C} \to \operatorname{Hom}_{{\mathfrak g}}\left({\mathcal F}_{\mu}\left(V_{0}\right),\bar{{\mathcal F}}_{\mu}\left(V_{0}\right)\right)={\mathbb C} \to
\notag\end{equation}
\begin{equation}
\operatorname{Hom}_{{\mathfrak g}}\left({\mathcal
    F}_{\mu}\left(V_{0}\right),\bar{{\mathcal
      F}}_{\mu}\left(V_{0}\right)\right)={\mathbb C} \to
\operatorname{Ext}_{\mathfrak g,\mathfrak h}^{1}\left({\mathcal F}_{\mu}\left(V_{0}\right),{\mathcal F}_{\mu}^{\left(k\right)}\left(V_{0}\right)\right)={\mathbb C} \to
\notag\end{equation}
\begin{equation}
\operatorname{Ext}_{\mathfrak g,\mathfrak h}^{1}\left({\mathcal
F}_{\mu}\left(V_{0}\right),\bar{{\mathcal
F}}_{\mu}\left(V_{0}\right)\right) \to
\operatorname{Ext}_{\mathfrak g,\mathfrak h}^{1}\left({\mathcal
F}_{\mu}\left(V_{0}\right),\bar{{\mathcal
F}}_{\mu}\left(V_{0}\right)\right) \to \dots \notag\end{equation}
and therefore the map
\begin{equation}
\varphi:\operatorname{Ext}_{\mathfrak g,\mathfrak h}^{1}\left({\mathcal
    F}_{\mu}\left(V_{0}\right),\bar{{\mathcal
      F}}_{\mu}\left(V_{0}\right)\right) \to
\operatorname{Ext}_{\mathfrak g,\mathfrak h}^{1}\left({\mathcal F}_{\mu}\left(V_{0}\right),\bar{{\mathcal F}}_{\mu}\left(V_{0}\right)\right)
\notag\end{equation}
is injective. But by our construction $ \varphi\left(c\right)=0$.
Thus, we obtain contradiction with our assumption that $c$ is
    non-trivial.
\end{proof}

Lemma ~\ref{lm26} and Theorem ~\ref{th5} imply the following

\begin{theorem} \label{th6}\myLabel{th6}\relax
Let $ \chi $ be a non-integral or singular central character and
$ \bar{\nu}\in{\mathfrak h}^{*}/Q $ be such that $ {\mathcal
C}^{\chi}_{\bar{\nu}} $ is not empty. Then
$ {\mathcal C}^{\chi}_{\bar{\nu}}  $ is equivalent to the category of
finite-dimensional modules over the algebra of power series in one variable.
\end{theorem}

\begin{corollary} \label{sing}\myLabel{sing}\relax  Let $ \chi $ be a non-integral or singular central character and
$ \bar{\nu}\in{\mathfrak h}^{*}/Q $ be such that $ {\mathcal
C}^{\chi}_{\bar{\nu}} $ is not empty. Every indecomposable module
in $ {\mathcal C}^{\chi}_{\bar{\nu}}  $ is isomorphic to $
{\mathcal F}_{\mu}^{\left(k\right)}\left(V_{0}\right) $ for some
nonnegative integer $k$.

\end{corollary}

\begin{proof} Every finite-dimensional $ {\mathbb
    C}\left[\left[\frac{\partial}{\partial u}\right]\right] $-module
  has trivial action of the maximal ideal of $ {\mathbb
    C}\left[\left[\frac{\partial}{\partial u}\right]\right]$. By
the  Jordan decomposition theorem
every finite-dimensional indecomposable $ {\mathbb
    C}\left[\left[\frac{\partial}{\partial u}\right]\right] $-module
is isomorphic to $ {\mathbb C}\left[\frac{\partial}{\partial u}\right]/\left(\frac{\partial^{k}}{\partial u^{k}}\right) $.\end{proof}

\section{The structure of the category
$ {\protect \mathcal C}^{\chi}_{\bar{\nu}}  $ for regular integral
$ \chi $ } \label{regint}

The goal of this section is to prove the following

\begin{theorem}\label{reg} Let $\mathfrak{g}=\mathfrak{sl}(n+1)$.
Every regular integral block of ${\mathcal C}$  is equivalent to the category
of locally nilpotent modules over the quiver $\mathcal Q_n$ (where $n$ is
the number of vertices)
$$
\xymatrix{\bullet \ar@(ul,ur)[]|{y} \ar@<0.5ex>[r]^x & \bullet
\ar@<0.5ex>[l]^x \ar@<0.5ex>[r]^y & \bullet \ar@<0.5ex>[l]^y
\ar@<0.5ex>[r]^x & \ar@<0.5ex>[l]^x... \ar@<0.5ex>[r] & \bullet
\ar@<0.5ex>[l] \ar@(ul,ur)[]|{}}
$$
with relations $xy=yx=0$.
\end{theorem}

Lemma~\ref{lm2} implies that if $ \chi $ is regular integral then
$ {\mathcal C}^{\chi}_{\bar{\nu}}  $ is equivalent to $ {\mathcal
C}^{0}_{\bar{\nu}_1}  $ for suitable $ \bar{\nu}_1 $. Thus we may
assume that $\chi = 0$. First we describe the simple objects in $
{\mathcal C}^{0}_{\bar{\nu}} $ following \S 11 in \cite{M}. For
our convenience we slightly change the description provided in
\cite{M} by using homogeneous coordinates instead of local
coordinates on ${\mathbb P}^n$. Let $ \mu\in{\mathbb C}^{n+1} $, $
\widehat{\Omega}^{k} $ be the space of $ k $-forms on $ {\mathbb
C}^{n+1} $, and
\begin{equation}
\widehat{\Omega}^{k}\left(\mu\right):=t^{\mu}{\mathbb
C}\left[t_{0}^{\pm1},\dots ,t_{n}^{\pm1}\right]\otimes_{{\mathbb
C}\left[t_{0},\dots ,t_{n}\right]}\widehat{\Omega}^{k},
\notag\end{equation}
\begin{equation}
\Omega^{k}\left(\mu\right):=\left\{\omega\in\widehat{\Omega}^{k}\left(\mu\right)
\mid L_{E}\left(\omega\right)=|\mu|\omega,
i_{E}\left(\omega\right)=0\right\}, \notag\end{equation} where $
i_{E} $ denotes the contraction with the Euler vector field $ E $,
and $ L_{E} $ denotes the Lie derivative. The space $
\Omega^{k}\left(\mu\right) $ is a $ {\mathfrak g} $-module with $
{\mathfrak g} $-action defined by the Lie derivative. In this
section we assume that all $ \mu_{i}\notin{\mathbb Z} $. Then $
\Omega^{k}\left(\mu\right) $ is a cuspidal module; and it is
simple if $ |\mu|\not=0 $.

In this section we assume $ |\mu|=0 $. Then the de Rham differential $
d:\Omega^{k}\left(\mu\right) \to \Omega^{k+1}\left(\mu\right) $ is
well defined as $ L_{E}=d\circ i_{E}+i_{E}\circ d=0 $. Furthermore,
it is not difficult to see that  $ \mu_{i}\notin{\mathbb Z} $
imply that the de Rham complex is exact. Let
\begin{equation}
L_{k}:=d\left(\Omega^{k-1}\left(\mu\right)\right) =
\operatorname{Ker} d\cap\Omega^{k}\left(\mu\right).
\notag\end{equation} The following two results are proven in
\cite{M}.

\begin{theorem} \label{th2}\myLabel{th2}\relax  $ L_{1},\dots ,L_{n} $ are all up to isomorphism simple objects in
$ {\mathcal C}^{0}_{\overline{\gamma ( \mu })} $.

\end{theorem}

\begin{lemma} \label{lm30}\myLabel{lm30}\relax  Let $ |\mu|=0 $. Then $ \Omega^{0}\left(\mu\right)\cong  L_{1} $, $ \Omega^{n}\left(\mu\right)\cong  L_{n} $. If $ k=1,\dots ,n-1 $, then
$ \Omega^{k}\left(\mu\right) $ is an indecomposable $ {\mathfrak g} $-module, i.e. the following exact sequence
\begin{equation}
0 \to L_{k} \to \Omega^{k}\left(\mu\right) \to L_{k+1} \to 0
\label{equ51}\end{equation}\myLabel{equ51,}\relax
does not split.

\end{lemma}

Our next step is to construct indecomposable injectives in
$\bar{{\mathcal C}}^0$. We will do it by applying translation functors
to injectives in singular blocks.

Following the construction of ${\mathcal F}_{\mu}^{(m)}$ in the
previous section, for an arbitrary $ m>0 $ define
\begin{equation}
\Omega^{k}\left(\mu\right)^{\left(m\right)}:=\Omega^{k}\left(\mu\right)\oplus
u\Omega^{k}\left(\mu\right)\oplus
u^{2}\Omega^{k}\left(\mu\right)\oplus\dots \oplus
u^{m}\Omega^{k}\left(\mu\right), \notag\end{equation} where $
u=\log \left(t_{0}\dots t_{n}\right) $. Then define
a module
\begin{equation}
\bar{\Omega}^{k}\left(\mu\right)={\mathbb
C}\left[u\right]\Omega^{k}\left(\mu\right) \notag\end{equation} in
$ \bar{{\mathcal C}} $. Then $\bar{\Omega}^{k}\left(\mu\right)$
has an obvious filtration
\begin{equation}
0\subset\Omega^{k}\left(\mu\right)\subset\Omega^{k}\left(\mu\right)^{\left(1\right)}\subset\Omega^{k}\left(\mu\right)^{\left(2\right)}\subset\dots
\subset\Omega^{k}\left(\mu\right)^{\left(m\right)}\subset\dots
\label{equ31}\end{equation}\myLabel{equ31,}\relax

For every object $ M $ in $ \bar{{\mathcal C}} $ and   a
finite-dimensional $ {\mathfrak g} $-module $ V $, the module $
M\otimes V $ is in $ \bar{{\mathcal C}}$ as well. Since the center
of $ U\left({\mathfrak g}\right) $ acts locally finitely on $ M $,
one can define $ M^{\chi_\lambda} $ as the subspace of $ M $ on which all
elements of the center lying in $ \operatorname{Ker} \chi_{\lambda} $
act locally nilpotently. The following is a well-known fact (see \cite{BG}).

\begin{lemma} \label{lm99}\myLabel{lm99}\relax  For every injective module $ M $ in $ \bar{{\mathcal C}} $ and a
finite-dimensional $\mathfrak g$-module $V$, the modules $
M\otimes V $, $ M^{\chi_\lambda}$, and  $ \left(M\otimes V\right)^{\chi_\lambda} $ are
injective.
\end{lemma}

\begin{proof} It is enough to show that
$M\otimes V $ and $M^{\chi_\lambda}$ are injective. The injectivity
of $M\otimes V $ follows from the isomorphism
\begin{equation}
\operatorname{Hom}_{{\mathfrak g}}\left(X,M\otimes V\right)\cong
\operatorname{Hom}_{{\mathfrak g}}\left(X\otimes V^{*},M\right) \notag\end{equation}
Since $\bullet \otimes V^{*}$ and $\operatorname{Hom}_{{\mathfrak g}}(\bullet,M)$ are both exact,
$\operatorname{Hom}_{{\mathfrak g}}\left(\bullet,M\otimes V\right)$ is
also exact.

The injectivity of $ M^{\chi_\lambda} $ follows from the fact that $M^{\chi_\lambda}$ is a
direct summand in $M$.
\end{proof}

\begin{lemma} \label{lm31}\myLabel{lm31}\relax  Let $ |\mu|=0 $. Then the modules
$ \Omega^{k}\left(\mu\right)^{\left(m\right)} $ and $
\bar{\Omega}^{k}\left(\mu\right) $ are indecomposable modules with
unique irreducible submodules. The same holds for any nontrivial
quotients of $ \Omega^{k}\left(\mu\right)^{\left(m\right)} $ and $
\bar{\Omega}^{k}\left(\mu\right) $ as well.

\end{lemma}

\begin{proof} We prove the statement for $ \Omega^{k}\left(\mu\right)^{\left(m\right)} $
by induction on $m$ using the filtration~\eqref{equ31}.
We reason as in the proof of Lemma~\ref{lm26}. It
suffices to prove the statement for $
\Omega^{k}\left(\mu\right)^{\left(1\right)} $.

Suppose that $ L $
is a simple submodule of $
\Omega^{k}\left(\mu\right)^{\left(1\right)} $ and $ L\not=L_{k} $.
Then $ L\cap\Omega^{k}\left(\mu\right)=0 $ by Lemma~\ref{lm30},
hence the image of $ L $ under the natural projection $
\Omega^{k}\left(\mu\right)^{\left(1\right)} \to
\Omega^{k}\left(\mu\right) $ is $ L_{k} $ (since
$\Omega^{k}\left(\mu\right)$ has only one simple submodule and it is
$L(k)$). This implies that $L(k)\oplus L=
L_{k}+uL_{k} $ is a submodule of $
\Omega^{k}\left(\mu\right)^{\left(1\right)} $, which, as one can
easily check, is not true.

Now let $M:=\Omega^{k}\left(\mu\right)^{\left(1\right)}/L_{k} $
and $p:\Omega^{k}\left(\mu\right)^{\left(1\right)} \to M $ be the
natural projection. Then $ p\left(L_{k+1}\right)\subset M $ is a
simple submodule. Suppose that there is another simple submodule $
L $. Then the image of $ L $ in $ \Omega^{k}\left(\mu\right) $
under the natural projection $ M \to \Omega^{k}\left(\mu\right) $
must be $ L_{k} $. This again implies that $ L_{k}+uL_{k} $ is a
submodule of $\Omega^{k}\left(\mu\right)^{\left(1\right)}$, which
leads to a contradiction. The cases $ k=1 $ and $ k=n $ are
similar to the general case.\end{proof}

\begin{corollary} \label{cor41}\myLabel{cor41}\relax  There exists a unique filtration
\begin{equation}
0=F^{0}\subset F^{1}\subset F^{2}\subset F^{3}\subset\dots
\notag\end{equation} of $ \bar{\Omega}^{k}\left(\mu\right) $ such
that all quotients $ F^{i}/F^{i-1} $ are simple. Furthermore,  $
F^{i}/F^{i-1}\cong  L_{1} $ if $ k=0 $, and $ F^{i}/F^{i-1}\cong
L_{n} $ if $ k=n $. If $ 1\leq k\leq n-1 $, then $
F^{i}/F^{i-1}\cong L_{k} $ for odd $ i $ and $ F^{i}/F^{i-1}\cong
L_{k+1} $ for even $ i $.

\end{corollary}

\begin{lemma} \label{lm34}\myLabel{lm34}\relax
$ \operatorname{Hom}_{{\mathfrak g}}\left(\bar{\Omega}^{k}\left(\mu\right),\bar{\Omega}^{l}\left(\mu\right)\right)=0 $
 if $ k\not=l $, and $ \operatorname{End}_{{\mathfrak g}}\left(\bar{\Omega}^{k}\left(\mu\right)\right)={\mathbb C}[\left[\frac{\partial}{\partial u}\right]] $.
\end{lemma}

\begin{proof} Let $ \phi\in\operatorname{Hom}_{{\mathfrak g}}\left(\bar{\Omega}^{k}\left(\mu\right),\bar{\Omega}^{l}\left(\mu\right)\right) $ and $ \phi\not=0 $. Then $ \operatorname{Im} \phi $ contains a
simple submodule $ L_{l}\subset\bar{\Omega}^{l}\left(\mu\right) $. Hence $ \bar{\Omega}^{k}\left(\mu\right) $ contains a simple subquotient isomorphic
to $ L_{l} $ and~\eqref{equ51} implies $ l=k $ or $ k+1 $. On the other hand, by Corollary~%
\ref{cor41}, $ \bar{\Omega}^{k}\left(\mu\right)/\operatorname{Ker}
\phi $ contains a simple subquotient isomorphic to $ L_{k-1} $.
Hence $ \bar{\Omega}^{l}\left(\mu\right) $ has a simple
subquotient isomorphic to $ L_{k-1} $. Therefore $
\operatorname{Hom}_{{\mathfrak
g}}\left(\bar{\Omega}^{k}\left(\mu\right),\bar{\Omega}^{l}\left(\mu\right)\right)\not=0
$ implies $ k=l $. To prove the second statement use
Corollary~\ref{cor41}. Since any endomorphism preserves
the filtration, $\operatorname{End}_{{\mathfrak
    g}}(\bar{\Omega}^{k}(\mu)) $ is the projective limit
  of $\operatorname{End}_{{\mathfrak
    g}}\left(F^m\right)=
\mathbb C[\frac{\partial}{\partial u}]/(\frac{\partial^m}{\partial u^m})$.
 {}\end{proof}

Let $ V $ be the span of the functions $ t_{0},t_{1},\dots ,t_{n}
$ and consider $V$ as the natural $ (n+1) $-dimensional $
{\mathfrak g} $-module. For $ k=1,\dots ,n $ we have the following
sequence
\begin{equation}
0 \to \Omega^{k}\left(\mu\right) \xrightarrow[]{\theta}
\Omega^{k-1}\left(\mu-\varepsilon_{0}\right)\otimes V
\xrightarrow[]{\sigma} \Omega^{k-1}\left(\mu\right) \to\text{ 0,}
\notag\end{equation} where $ \theta=\sum
i_{\frac{\partial}{\partial t_{i}}}\otimes t_{i} $ and $
\sigma=\sum t_{i}\otimes\frac{\partial}{\partial t_{i}} $.
Obviously $ \theta$ and $\sigma $ are $ {\mathfrak g}
$-equivariant. The direct computation shows that $
\sigma\circ\theta=i_E =0 $. Furthermore, assume that $ |\mu|=0 $
and  consider the component of the above exact sequence
corresponding to the trivial generalized central character.
The resulting sequence is
\begin{equation}
0 \to \Omega^{k}\left(\mu\right) \xrightarrow[]{\varphi} S^{k}
\xrightarrow[]{\psi} \Omega^{k-1}\left(\mu\right) \to 0,
\label{equ33}\end{equation}\myLabel{equ33,}\relax where $
S^{k}:=\left(\Omega^{k-1}\left(\mu-\varepsilon_{0}\right)\otimes
V\right)^{\chi_0}.$

\begin{lemma} \label{lm32}\myLabel{lm32}\relax  The sequence ~\eqref{equ33} is exact for $
k=1,\dots ,n $. Moreover, $ S^{k} $ is an indecomposable module
with unique simple submodule and unique simple quotient, both
isomorphic to $ L_{k} $.

\end{lemma}

\begin{proof} Note that $ \Omega^{k}\left(\mu\right)\simeq{\mathcal F}_{\mu-k\varepsilon_{0}}\left(V(\varepsilon_{1}+\dots +\varepsilon_{k})\right) $,
where $ V(\eta) $ is the irreducible $ P $-module with highest weight $ \eta $. By Remark~%
\ref{rem99},
$$ \Omega^{k-1}\left(\mu-\varepsilon_{0}\right)\otimes
V\cong {\mathcal
F}_{\mu-k\varepsilon_{0}}\left(V(\varepsilon_{1}+\dots
+\varepsilon_{k-1})\otimes V\right). $$ Now use the exact sequence
of $ P $-modules
\begin{equation}
0 \to V(2\varepsilon_{1}+\dots +\varepsilon_{k-1})\oplus
V(\varepsilon_{1}+\dots +\varepsilon_{k}) \to
V(\varepsilon_{1}+\dots +\varepsilon_{k-1})\otimes V \to
V(\varepsilon_{0}+\varepsilon_{1}+\dots +\varepsilon_{k-1}) \to 0
\notag\end{equation} which induces the following exact sequence of
$ {\mathfrak g} $-modules
\begin{eqnarray*}
0 & \to &  {\mathcal
F}_{\mu-k\varepsilon_{0}}\left(V(2\varepsilon_{1}+\dots
+\varepsilon_{k})\right)\oplus{\mathcal
F}_{\mu-k\varepsilon_{0}}\left(V(\varepsilon_{1}+\dots
+\varepsilon_{k})\right) \to \\ & \to & {\mathcal
F}_{\mu-k\varepsilon_{0}}\left(V(\varepsilon_{1}+\dots
+\varepsilon_{k+1})\right)\otimes V  \to {\mathcal
F}_{\mu-k\varepsilon_{0}}\left(V(\varepsilon_{0}+\varepsilon_{1}+\dots
+\varepsilon_{k-1})\right) \to\text{ 0.} \notag\end{eqnarray*} But
\begin{equation}
{\mathcal F}_{\mu-k\varepsilon_{0}}\left(V(\varepsilon_{0}+\varepsilon_{1}+\dots +\varepsilon_{k-1})\right)=\Omega^{k-1}\left(\mu\right)\text{, }{\mathcal F}_{\mu-k\varepsilon_{0}}\left(V(\varepsilon_{1}+\dots +\varepsilon_{k})\right)=\Omega^{k}\left(\mu\right),
\notag\end{equation}
\begin{equation}
\left({\mathcal F}_{\mu-k\varepsilon_{0}}\left(V(2\varepsilon_{1}+\dots +\varepsilon_{k})\right)\right)^{\chi_0}=0.
\notag\end{equation}
Hence~\eqref{equ33} is an exact sequence.

Now we will show that $ S^{k} $ has a unique irreducible submodule
isomorphic to $ L_{k} $ (this will imply the indecomposability of
$ S^{k} $ as well). Since the functor ${}^{\vee}$ preserves tensor
products with finite-dimensional modules and maps a simple module
to itself, the irreducibility of $
\Omega^{k-1}\left(\mu-\varepsilon_{0}\right) $ implies $
({S}^{k})^{\vee}=S^{k} $. By Lemma~\ref{lm30}, $
\Omega^{k}\left(\mu\right) $ and $ \Omega^{k-1}\left(\mu\right) $
are indecomposable. If $ S^{k} $ has another irreducible submodule
then, by the indecomposability of $ \Omega^{k}\left(\mu\right) $,
this submodule is isomorphic to $ L_{k-1} $. But then since $
(S^{k})^{\vee}=S^{k} $, $ S^{k} $ must have an irreducible quotient
isomorphic to $ L_{k-1} $, which is impossible due to the
indecomposability of $ \Omega^{k-1}\left(\mu\right) $. Finally,
again by duality, $ S^{k} $ has a unique irreducible quotient
isomorphic to $ L_{k} $.\end{proof}

Recall that $ \Omega^{k-1}\left(\mu-\varepsilon_0\right) $ is a
simple cuspidal module with singular central character
$\chi_{-k\varepsilon_0}$, and
$\bar{\Omega}^{k-1}\left(\mu-\varepsilon_{0}\right)$ is an indecomposable
  injective in this singular block.
Set
$I^{k}:=\left(\bar{\Omega}^{k-1}\left(\mu-\varepsilon_{0}\right)\otimes
V\right)^{\chi_0} $ for $ k=1,\dots ,n $. The exact
sequence~\eqref{equ33} leads to the following exact sequence
\begin{equation}
0 \to \bar{\Omega}^{k}\left(\mu\right) \xrightarrow[]{i_{k}} I^{k} \xrightarrow[]{p_{k}} \bar{\Omega}^{k-1}\left(\mu\right) \to 0
\label{equ34}\end{equation}\myLabel{equ34,}\relax

\begin{lemma} \label{lm33}\myLabel{lm33}\relax The module $ I^{k} $ is an injective object in $ \bar{{\mathcal C}}^{0}_{\overline{\gamma\left(\mu\right)}} $,
and it has a unique simple submodule, which is isomorphic to $L_{k} $.
\end{lemma}

\begin{proof} The injectivity of $ I^{k} $ follows from
  Lemma~\ref{lm99}.
To prove that $ I^{k} $  has a unique simple submodule isomorphic to $L^{k}$
recall that $T_\chi^\eta\circ T_\eta^\chi=\mbox{Id}\oplus\mbox{Id}$ if
$\eta$ is singular and $\chi$ is regular (see \cite{BG}). In our case
$\eta=\chi_{-k\varepsilon_0}$ and $\chi=\chi_0$.
The exact sequence ~\eqref{equ33} implies that
$$(L_k\otimes V^*)^{\chi_{-k\varepsilon_0}}=\Omega^{k-1}(\mu-\varepsilon_0)$$
and
$$(L_i\otimes V^*)^{\chi_{-k\varepsilon_0}}=0$$
if $i\neq k$.
Thus, we have
$$\mbox{Hom}_{\mathfrak g}(L_k,I_k)=\mbox{Hom}_{\mathfrak
  g}(L_k,(\bar{\Omega}^{k-1}(\mu-\varepsilon_0)\otimes V)^{\chi_0})=\mbox{Hom}_{\mathfrak
  g}((L_k\otimes V^{*})^{\chi_{-k \varepsilon_0}},\bar{\Omega}^{k-1}(\mu-\varepsilon_0))=\mathbb C$$
and for $i\neq k$
$$\mbox{Hom}_{\mathfrak g}(L_i,I_k)=\mbox{Hom}_{\mathfrak
  g}(L_i,(\bar{\Omega}^{k-1}(\mu-\varepsilon_0)\otimes V)^{\chi_0})=\mbox{Hom}_{\mathfrak
  g}((L_i\otimes V^{*})^{\chi_{-k \varepsilon_0}},\bar{\Omega}^{k-1}(\mu-\varepsilon_0))=0.$$
\end{proof}

\begin{corollary} \label{cor34}\myLabel{cor34}\relax  $ \operatorname{Hom}_{{\mathfrak g}}\left(\bar{\Omega}^{k}\left(\mu\right),I^{l}\right)=0 $
 if $ k\not=l,l-1 $, and $ \operatorname{Hom}_{{\mathfrak g}}\left(I^{k},I^{l}\right)=0 $ if
$ k\not=l,l\pm1 $.

\end{corollary}

\begin{proof} The statements follow from~\eqref{equ34} and Lemma~\ref{lm34}.{}\end{proof}

Using~\eqref{equ34}, for $ k=1,...,n-1 $, define $
\psi_{k}\in\operatorname{Hom}_{{\mathfrak
g}}\left(I^{k+1},I^{k}\right) $ by setting $ \psi_{k}:=i_{k}\circ
p_{k+1} $.

Corollary~\ref{cor41} implies that $
\bar{\Omega}^{k-1}\left(\mu\right)/L_{k-1} $ has a submodule
isomorphic to $ L_{k} $. Since $ I^{k} $ is injective and has a
submodule isomorphic to $ L_{k} $, there is a homomorphism $
s_{k}\colon \bar{\Omega}^{k-1}\left(\mu\right)/L_{k-1} \to I^{k}
$. Using the exact sequence~\eqref{equ34} and
Corollary~\ref{cor41},  one can easily prove the existence of an
exact sequence
\begin{equation}
0 \to \bar{\Omega}^{k-1}\left(\mu\right)/L_{k-1}
\xrightarrow[]{s_{k}} I^{k} \xrightarrow[]{t_{k}}
\bar{\Omega}^{k}\left(\mu\right)/L_{k} \to\text{ 0.}
\notag\end{equation} We assume that $ L_{0}=0 $, so that the above
exact sequence is valid for $ k=1 $.

Define $ \varphi_{k}\in\operatorname{Hom}_{{\mathfrak
g}}\left(I^{k},I^{k+1}\right) $ by $ \varphi_{k}:=s_{k+1}\circ
t_{k} $. It is not hard to verify that
\begin{equation}
\varphi_{k+1}\circ\varphi_{k}=\psi_{k}\circ\psi_{k+1}=0.
\notag\end{equation} Finally, introduce $
\xi\in\operatorname{End}_{{\mathfrak g}}\left(I^{1}\right) $ by $
\xi:=s_{1}\circ p_{1} $ and $
\eta\in\operatorname{End}_{{\mathfrak g}}\left(I^{n}\right) $ by $
\eta:=i_{n}\circ t_{n} $ (for the latter we use that $
\bar{\Omega}^{n}\left(\mu\right)/L_{n}\cong
\bar{\Omega}^{n}\left(\mu\right) $). One can check that
\begin{equation}
\xi\circ\psi_{1}=\varphi_{1}\circ\xi=\psi_{n-1}\circ\eta=\eta\circ\varphi_{n-1}=0.
\notag\end{equation} Let $ I:=I^{1}\oplus\dots \oplus I^{n} $ and
$ {\mathcal E}:=\operatorname{End}_{{\mathfrak g}}\left(I\right)
$. Let $ e_{1},\dots ,e_{n} $ be the standard idempotents in $
{\mathcal E} $ and let $ {\mathcal R} $ be the radical of $
{\mathcal E} $. Then $ {\mathcal R} $ defines a filtration of $I$
\begin{equation}
0\subset{\mathcal R}^{1}\left(I\right)\subset{\mathcal
R}^{2}\left(I\right)\subset\dots , \notag\end{equation} such that
$ {\mathcal R}^{m}\left(I\right)=\operatorname{Ker} {\mathcal
R}^{m} $. The quotients
$\mathcal S^m(I):={\mathcal R}^{m}\left(I\right)/{\mathcal
R}^{m-1}\left(I\right) $ are semisimple over $\mathcal E$ and
therefore over $\mathfrak g$ (see Theorem ~\ref{th5}). Moreover,
$\mathcal
S^{1}\left(I\right)\simeq L_{1}\oplus L_{2}\oplus\dots \oplus L_{n} $.

\begin{lemma} \label{lm100}\myLabel{lm100}\relax  $ I^{k}/L_{k}\cong \left( \bar{\Omega}^{k}\left(\mu\right) /L_{k}\right) \oplus\bar{\Omega}^{k-1}\left(\mu\right) $,
for $k=1,...,n$.

\end{lemma}

\begin{proof} The exact sequence ~\eqref{equ34} leads to the following exact sequence
\begin{equation}
0 \to \bar{\Omega}^{k}\left(\mu\right)/L_{k} \xrightarrow[]{w_{k}}
I^{k}/L_{k} \xrightarrow[]{u_{k}}
\bar{\Omega}^{k-1}\left(\mu\right) \to\text{ 0.}
\label{equ201}\end{equation}\myLabel{equ201,}\relax We will show
that (\ref{equ201}) splits. Recall that $
s_{k}:\bar{\Omega}^{k-1}\left(\mu\right)/L_{k-1} \to I^{k} $ is an
injection. Using  Corollary~\ref{cor41} we obtain the following
 exact sequence
\begin{equation}
0 \to L_{k} \to \bar{\Omega}^{k-1}\left(\mu\right)/L_{k-1} \to \bar{\Omega}^{k-1}\left(\mu\right) \to\text{ 0.}
\notag\end{equation}
Therefore one can construct an injective map
\begin{equation}
v_{k}:\bar{\Omega}^{k-1}\left(\mu\right)\cong
\left(\bar{\Omega}^{k-1}\left(\mu\right)/L_{k-1}\right)/L_{k} \to
I^{k}/L_{k}. \notag\end{equation} We claim that the composite $
u_{k}\circ v_{k}:\bar{\Omega}^{k-1}\left(\mu\right) \to
\bar{\Omega}^{k-1}\left(\mu\right) $ is an isomorphism. We first
note that $ u_{k}\circ v_{k} $ is injective for $ k\leq n-1 $.
Indeed, we have that $ v_{k} $ is injective, $ \operatorname{Ker}
u_{k} $ has a unique simple submodule isomorphic to $ L_{k+1} $,
 and all irreducible constituents of $
\operatorname{Im} v_{k} $ are isomorphic to $ L_{k} $ or $ L_{k-1}
$. Hence, $ \operatorname{Ker} u_{k} $ intersects $
\operatorname{Im} v_{k} $ trivially. In the case $ k=n $, we
notice that $ \operatorname{Ker} u_{n} $ has a unique simple
submodule isomorphic to $ L_{n} $ and $ \operatorname{Im} v_{n} $
has a unique maximal submodule isomorphic to $ L_{n-1} $. Thus,
again $ \operatorname{Im} v_{n}\cap\operatorname{Ker} u_{n}= 0$.
Now note that $ \operatorname{Im} (u_{k}\circ v_{k}) $ has
infinite length, hence, by Corollary~\ref{cor41}, $ u_{k}\circ v_{k}
$ is surjective. Therefore,~\eqref{equ201} splits.\end{proof}

\begin{corollary} \label{cor101}\myLabel{cor101}\relax  If $ k>1 $, then $ {\mathcal S}^{k}\left(I\right)\cong  L_{1}^{\oplus2}\oplus\dots \oplus L_{n}^{\oplus2} $.

\end{corollary}

With a slight abuse of notations we will denote the images of $
\varphi_{i},\psi_{i},\xi $ and $ \eta $ under the natural
projection $ {\mathcal E} \to {\mathcal E}/{\mathcal R}^{m} $ by
the same letters.

\begin{theorem} \label{th102}\myLabel{th102}\relax  The set $\{ \xi,\eta,\varphi_{1},\dots ,\varphi_{n-1},\psi_{1},\dots ,\psi_{n-1}\}$
forms a basis of $ {\mathcal R}/{\mathcal R}^{2} $ and generates $
{\mathcal R}/{\mathcal R}^{m} $ for any $ m>0 $.

\end{theorem}

\begin{proof} It is clear from their construction that
$ \xi,\eta,\varphi_{1},\dots ,\varphi_{n-1},\psi_{1},\dots
,\psi_{n-1} $ are linearly independent. Since $ \dim  {\mathcal
R}/{\mathcal R}^{2}=2n $, by Corollary~\ref{cor101}, $
\xi,\eta,\varphi_{1},\dots ,\varphi_{n-1},\psi_{1},\dots
,\psi_{n-1} $ form a basis of $ {\mathcal R}/{\mathcal R}^{2} $.
Similarly, $ \dim  {\mathcal R}^{m}/{\mathcal R}^{m+1}=2n $ for
 $ m>1 $ as well. Set $ \phi_{k}:=\psi_{k}\circ\varphi_{k} $ for $
k=1,\dots ,n-1 $, and
$\bar{\phi}_{k}:=\varphi_{k-1}\circ\psi_{k-1} $ for $ k=2,\dots ,n
$. Then using Corollary \ref{cor101} and dimension calculations we
verify that $ \xi^{2p},\eta^{2p} $, $ \phi_{1}^{p},\dots
,\phi_{n-1}^{p} $, $ \bar{\phi}_{2}^{p},\dots ,\bar{\phi}_{n}^{p}
$ form a basis of $ {\mathcal R}^{2p}/{\mathcal R}^{2p+1} $, and $
\xi^{2p+1},\eta^{2p+1} $, $ \varphi_{1}\circ\phi_{1}^{p},\dots
,\varphi_{n-1}\circ\phi_{n-1}^{p} $, $
\psi_{1}\circ\bar{\phi}_{2}^{p},\dots
,\psi_{n-1}\circ\bar{\phi}_{n}^{p} $ form a basis of $ {\mathcal
R}^{2p+1}/{\mathcal R}^{2p+2} $. \end{proof}

Let $ {\mathcal B} $ denote the subalgebra in $ {\mathcal E} $
generated by $ e_{1},\dots ,e_{n} $, $ \xi,\eta,\varphi_{1},\dots
,\varphi_{n-1},\psi_{1},\dots ,\psi_{n-1} $. Denote by $ {\mathcal
D} $ the subcategory of finite-dimensional $ {\mathcal B}
$-modules on which $ {\mathcal R}\cap{\mathcal B} $ acts
nilpotently.

\begin{corollary} \label{cor103}\myLabel{cor103}\relax  Let $ |\mu|=0 $ and $ \mu_{i}\notin{\mathbb Z} $ for all $ i=0,\dots ,n $. Then the category
$ {\mathcal C}^{0}_{\overline{\gamma\left(\mu\right)}} $ is equivalent
to the category $ {\mathcal D} $.

\end{corollary}

\begin{proof} As follows from Theorem~\ref{th102}, $ {\mathcal B}/\left({\mathcal R}^{m}\cap{\mathcal B}\right)\cong {\mathcal E}/{\mathcal R}^{m} $. Since a
finite-dimensional module over $ {\mathcal E} $ is a module over $
{\mathcal E}/{\mathcal R}^{m} $ for some $ m $, the category $
{\mathcal D} $ coincides with the category of finite-dimensional $
{\mathcal E} $-modules. Thus, the statement follows from
Theorem~\ref{th5}. {}\end{proof}

The correspondence
\begin{eqnarray*} \sum_{i=1}^{\frac{n}{2}}
(\varphi_{2i-1} + \psi_{2i-1}) \mapsto x;&  \; \; \xi +
\displaystyle \sum_{i=1}^{\frac{n}{2}-1} (\varphi_{2i} + \psi_{2i}) + \eta
\mapsto y,& \mbox{ for even }n\\
\sum_{i=1}^{\frac{n-1}{2}} (\varphi_{2i-1} + \psi_{2i-1}) + \eta
 \mapsto  x; & \; \; \xi + \displaystyle  \sum_{i=1}^{\frac{n-1}{2}}
(\varphi_{2i} + \psi_{2i}) \mapsto y,& \mbox{ for odd }n
\end{eqnarray*}
establishes an equivalence of the category $\mathcal D$ and the
locally nilpotent representations
of the quiver $\mathcal Q_n$. Hence
Theorem \ref{reg} is proven.
\begin{remark} \label{Kosz}\myLabel{Kosz}\relax
Note that $\mathcal B$ is a graded quadratic algebra. It is natural to
ask if $\mathcal B$ is Koszul (in the sense of \cite{BGS}). We conjecture
that the answer to this question is positive. A strong indication  that this conjecture is true is that
the numerical criterion (Lemma
2.11.1 \cite{BGS}) holds. Indeed, the matrix $B(t)=P(\mathcal B,t)$ is
given by the formula
$$b_{ij}=0\,\, \text{if}\,\, |i-j|>1,$$
$$b_{ij}=\frac{t}{1-t^2}\,\, \text{if} \,\,|i-j|=1,$$
$$b_{ii}=\frac{1+t^2}{1-t^2}\,\, \text{if} \,\,i\neq 1,n,$$
$$b_{11}=b_{nn}=\frac{1+t+t^2}{1-t^2}.$$
The matrix $C(t)=P(\mathcal B^!,t)$ is a symmetric matrix defined by
$$c_{ij}=\frac{t^{j-i}(1+t^{2i-1})(1+t^{2(n-j)+1})}{1-t^{2n} }\,\,
\text{if}\,\,i\leq j.$$
Then one can check that
$$P(\mathcal B,t)P^t(\mathcal B^!,-t)= 1.$$
\end{remark}

\section{Explicit description of the indecomposable objects in
$ {\protect \mathcal C}^{\chi}_{\bar{\nu}} $ for regular integral $ \chi
$} \label{explicit}

In this section we parameterize and explicitly describe all indecomposable
objects in the category ${\mathcal C}^{\chi}_{\bar{\nu}}$. In particular we show that every indecomposable can be obtained by applying natural combinatorial operations (gluing and polymerization) to subquotients of the modules $\Omega^k (\mu) ^{(m)}$.

\subsection{The quiver ${\mathcal Q}_n$ and its indecomposable representations } \label{linalg}

The quiver ${\mathcal Q}_n $ defines a special biserial algebra the theory of which is well established. The classification of the indecomposable representations of special biserial algebras is usually attributed to Gelfand-Ponomarev. In \cite{GP} they considered a special case  but the proof can be generalized as shown in \cite{Kh} and \cite{R}. A good overview of the representation theory of special biserial quivers can be found for example in \cite{Er}.  In what follows we describe the two types of indecomposable objects of   ${\mathcal Q}_n $ following \cite{GP} and \cite{Kh}.

First, we note that the problem of classifying indecomposable
representations of ${\mathcal Q}_n $ is equivalent to the similar
problem for two linear operators $x$ and
$y$ in a graded vector space $V=V_1\oplus\dots\oplus V_n$ satisfying
the conditions

(C0) $x(V_i)\subseteq V_{\pi_1(i)}$, $y(V_i)\subseteq
V_{\pi_2(i)}$, where $\pi_1$ and $\pi_2$ are the permutations $(12)(34)...$ and $(23)(45)...$, respectively;

(C1) $xy=yx = 0$;

(C2) $x$ and $y$ are nilpotent.

\medskip

{\it Strings.} The first type of indecomposable representations of ${\mathcal
Q}_n$ is parameterized by string quivers. More precisely, a {\it
string quiver} is a quiver $Q = (Q_v, Q_a)$, for which $Q_v=\{
v_1,...,v_k\}$, there are no arrows between $v_i$ and $v_j$ for $|i-j| \geq 2$, and for every $i$, $2\leq i\leq k$, there is
exactly one arrow connecting $v_{i-1}$ and $v_i$. The arrow connecting
$v_{i-1}$ and $v_i$ will be denoted by
$\overrightarrow{v_{i-1}v_i}$ if its head is $v_{i-1}$, and by  $\overleftarrow{v_{i-1}v_{i}}$, otherwise.
To each vertex $v_i \in Q_v$ we attach a  {\it label}  $l(v_i)$, $1\leq
l(v_i) \leq n$, such that: $l(v_i) = \pi_1(l(v_{i-1}))$ whenever
$\overrightarrow{v_{i-1}v_i} \in Q_a$, and $l(v_i) =
\pi_2(l(v_{i-1}))$ whenever $\overleftarrow{v_{i-1}v_{i}} \in
Q_a$. A {\it graded string} $S$ is a pair $(Q, l)$ of a string
quiver $Q$ and a labelling $l: Q_v \to \{ 1,...,n\}$ compatible with the
grading condition (C0) as described above.

\begin{example} \label{exstring}
let $n=3$ and $k=6$. An example of a graded string is pictured
below. The numbers above the vertices are the labels of the
string.
$$
\xymatrix@1{{\bullet}\ar[r]^<{2} & {\bullet} &
\ar[l]_>{1}{\bullet} \ar[r]^<{1} & {\bullet} \ar[r]^<{2} &
{\bullet} \ar[r]^>(1.27){2}^<{1}& {\bullet}}
$$
\end{example}

Every graded string $S= (Q, l)$ determines in a natural way a {\it
string representation} $I(S)$ of ${\mathcal Q}_n$:
we attach one-dimensional spaces ${\mathbb C} e_i$ to each vertex
$v_i$, and set $I(S)= \bigoplus_{j=1}^n V_j$, where
$V_j:=\bigoplus_{l(v_i)=j} {\mathbb C} e_i$. We also set  $x(e_i)
= e_{i+1}$ if $\overrightarrow{v_{i}v_{i+1}} \in Q_a$ and
$y(e_{i+1}) = e_{i}$ if $\overleftarrow{v_{i}v_{i+1}} \in Q_a$.
All remaining $x(e_j)$ and $y(e_j)$ are zeroes.

{\it Bands.} We consider quivers $Q$ whose sets of
vertices $Q_v = \{v_1,...,v_k\} $ form a regular $k$-polygon
($k>2$); the adjacent vertices $v_{i-1}$ and $v_i$ are connected
by exactly one arrow. As in the case of strings we denote this arrow by either $\overrightarrow{v_{i-1}v_{i}}$ or
$\overleftarrow{v_{i-1}v_{i}}$ (we set $v_{k+1} := v_1$). We
again define labelling $l : Q_v
\to \{ 1, ..., n\}$  compatible with the grading condition (C0)
with compatibility conditions identical to those for the
strings. Such labelled quivers will be called {\it graded
polygons}. Note that every graded polygon $P$ can be ``unfolded''
at a sink $v_i$ to a graded string $S(P, v_i)$. Namely,  the
graded string $S(P, v_i)$ has a set of vertices $\{v_i, ...,
v_1,v_2,.. v_{i-1}, \bar{v}_i\}$ and the same arrows as $P$,
except that the arrow $\overrightarrow{v_{i-1}v_i}$ or
$\overleftarrow{v_{i-1}v_{i}}$ is replaced by
$\overrightarrow{v_{i-1}\bar{v}_i}$ or
$\overleftarrow{v_{i-1}\bar{v}_i}$, respectively.

For a graded directed\footnote{Notice that a graded polygon is
directed if and only if not all arrows go clockwise or
counterclockwise.} polygon $P = (Q, l)$ with $k$ vertices, a
nonzero complex number $\lambda$, and a positive integer $r$ we
define the {\it band representation} $I(P, \lambda, r)$ of
${\mathcal Q}_n$ as follows. To each vertex $v_i$ we
attach an $r$-dimensional vector space $U_i$, and set $I(P,
\lambda, r) := \bigoplus_{j=1}^n V_j$, where $V_j :=
\bigoplus_{l(v_i)=j} U_i$. We define $x : U_i \to U_{i+1}$ and $y
: U_{i+1} \to U_{i}$  to be isomorphisms whenever
$\overrightarrow{v_iv_{i+1}} \in Q_a$ and
$\overleftarrow{v_{i}v_{i+1}} \in Q_a$, respectively (in other
words, if the vertices are numbered clockwise, then the $x$'s are
directed clockwise, while the $y$'s - counterclockwise). In all
other cases $x_{|U_j} = 0$ and $y_{|U_j} = 0$. In addition, we
require that the matrix of the composition $d_k...d_2d_1: U_1 \to
U_1$ of all isomorphisms $d_i: U_i \to U_{i+1}$, $d_i = x$ or
$y^{-1}$, is represented by a single Jordan block $J_r(\lambda)$
in an appropriate basis of $U_1$ with non-zero eigenvalue
$\lambda$. It is easy to check that $d_i...d_1d_k...d_{i-1}$ and
$d_k...d_2d_1$ are similar, and thus the choice of $U_1$ is
irrelevant, i.e. $I(P, \lambda, r)$ is indeed uniquely determined
by the triple $(P,\lambda, r )$.

One should note that not every band representation $I(P, \lambda, r)$ is
indecomposable. In order $I(P, \lambda, r)$ to be
indecomposable, $P$ has to be a graded directed polygon  with no
rotational symmetry (a rotational symmetry of $P$ is a rotation
of the plane which preserves the quiver and the
labels of $P$).

\begin{proposition} (\cite{GP, Kh})
Every indecomposable representation of ${\mathcal Q}_n$ is
isomorphic either to a string module $I(S)$ for some (unique)
graded string $S$, or to a band module $I(P, \lambda, r)$ for
some (unique) triple $(P, \lambda, r)$, where  $P$ is a graded
directed polygon with no rotational symmetry,  $\lambda$ is a
nonzero complex number, and $r$ is a positive integer.
\end{proposition}

\subsection{Operations on strings and bands} \label{glue}
In this subsection we introduce three operations on the set of
$\mathfrak g$-modules which help us to describe the string and
band modules in an alternative way. Namely, we reduce the case of
a general indecomposable representation of ${\mathcal Q}_n$
to the case of a  string whose arrows have the same direction. We
follow the terminology and notation of \cite{GGP}.

{\it Gluing.} Let $A_1$ and $A_2$ be vector spaces,
$A_1'$ and $A_2'$ be isomorphic subspaces of $A_1$ and $A_2$,
respectively, and let $\sigma : A_1' \to A_2'$ be an isomorphism.
Set $D_{\sigma}:=\{(a, \sigma(a))\; | \; a\in A_1' \}$. Then the
{\it gluing of $A_1$ and $A_2$ relative to $\sigma$} is the
quotient space $A_1 \oplus A_2 /D_{\sigma}$. The notion of gluing
easily extends for ${\mathfrak g}$-modules. In the case of
$1$-dimensional spaces (respectively, simple {$\mathfrak
g$}-modules) $A_1'$ and $A_2'$,  the gluing $A_1 \oplus A_2
/D_{\sigma}$ does not depend on a choice of $\sigma$, because
$A_1 \oplus A_2 /D_{\sigma} \cong A_1 \oplus A_2 /D_{\lambda
\sigma}$ for any nonzero $\lambda \in {\mathbb C}$. In such cases
we will write $A_1 \oplus A_2 /D(A_1')$ (note that $D(A_1')$ is the
diagonal embedding of $A_1'\cong A_2'$ in $A_1 \oplus A_2$).

{\it Dual gluing.} In the dual setting we start with two pairs of
spaces $A_1'\subseteq A_1$ and $A_2'\subseteq A_2$ and an
isomorphism $\sigma: A_1/A_1' \to A_2/A_2'$. Then {\it the dual
gluing of $A_1$ and $A_2$ relative to $\sigma$} is the subspace
$\{(a_1, a_2)\in A_1 \oplus A_2 \; | \; \sigma(\bar{a}_1) =
\bar{a}_2 \}$ of $A_1 \oplus A_2$ (here $\bar{a}_i = a_i + A_i'$).
In the case of weight ${\mathfrak g}$-modules we will use the
following alternative form of the dual gluing of $A_1$ and $A_2$:
$\left( A_1^{\vee} \oplus A_2^{\vee} /D_{\sigma}\right)^{\vee}$.

{\it Polymerization.} Let $A$ be a vector space and $A_1 \neq A_2$
be two isomorphic subspaces of $A$. Fix an isomorphism $\sigma :
A_1 \to A_2$. The {\it polymerization of $p$ copies of $A$
relative to $\sigma$} is by definition the vector space
$A^{(p)}(\lambda, \sigma):=A^{\oplus p}/A_{\sigma}^{\lambda}$
where $A_{\sigma}^{\lambda}$ is the submodule of $A^{\oplus p}$
consisting of $(\sigma(a_1) - \lambda a_1, \sigma(a_2)
- \lambda a_2 - a_1,..., \sigma(a_p) - \lambda a_p - a_{p-1})$.
The polymerization is also well defined for a ${\mathfrak
g}$-module $A$ and  two isomorphic submodules $A_1$ and
$A_2$ of $A$.

Notice that for a graded string $S$ with a set of vertices $\{v_1,...,v_k\}$, $I(S)$ can be obtained by gluing $I(S_1)$ and $I(S_2)$ at a sink $v_i$ where $S_1$ and $S_2$ have sets of vertices $\{ v_1,...,v_i\}$ and  $\{ v_{i+1},...,v_n\}$, respectively. We similarly represent $I(S)$ by
dual gluing of  $I(S_1)$ and $I(S_2)$ at a source $v_i$. In both cases we will write $S= S_1S_2$. We also may represent a band representation $I(P, \lambda, r)$ as a polymerization of the unfolded graded string representation $I(S(P, v_i))$ for any sink $v_i$ of $P$.

\subsection{Explicit description of the indecomposables in  ${\mathcal C}^{\chi}_{\bar{\nu}} $}

As
discussed in the beginning of the previous section, we may restrict
our attention to the category ${\mathcal
C}^0_{\overline{\gamma(\mu})}$ for  $|\mu|=0$ and $\mu_i \notin
\mathbb Z$.

In order to describe explicitly all indecomposables in ${\mathcal
C}^{0}_{\overline{\gamma(\mu})}$ we use Theorem \ref{reg} and
combine it with results in  \S \ref{linalg}, and
\S \ref{glue}. We provide the description in three steps. With
small abuse of notation we will denote the indecomposable objects
(defined up to an isomorphism) of ${\mathcal
C}^{0}_{\overline{\gamma(\mu})}$ by $I(S)$ and $I(P, \lambda,
r)$ as well.

{\it Homogeneous strings.} Here we list all indecomposables $I(S)$
that correspond to {\it homogeneous} graded strings $S$, i.e. such
that all arrows have the same direction. There are exactly two
homogeneous graded strings with $m$ vertices whose leftmost vertex
is labelled by $s$, $1\leq s \leq n$: the string $X_m(s)$ where
all arrows go in the left-to-right direction; and the string
$Y_m(s)$ where all arrows go in the right-to-left direction. With the aid of Lemma \ref{lm31} and Corollary \ref{cor41} and using a case-by-case verification we easily find an explicit realization of $I(X_m(s))$ and $I(Y_m(s))$ as subquotients of $\Omega^{k}(\mu)^{(r)}$ for suitable $k$ and $r$. Alternatively, we may use quotients of  $\Omega^{k}(\mu)^{(r)}$  and  $\left( \Omega^{k}(\mu)^{(r)} \right)^{\vee}$. For example if $m$ is even and $s$ is odd, then $I(X_m(s))\cong \left(\Omega^{s}(\mu)^{\left(\frac{m}{2} + 1\right)}\right)^{\vee}/L_{s}$. The details are left to the reader.

\medskip

{\it Arbitrary strings.} In the case of an arbitrary graded string
$S$ we first represent $S$ as a product  of homogeneous strings and then apply gluing and dual gluing to find $I(S)$.
More explicitly, if $S = X_m(s)S'$ (or, respectively, $S=Y_m(s)S'$), where $S'$ is a graded
string with a right-to-left (respectively, left-to-right) leftmost
arrow, then
\begin{eqnarray*}
I(X_m(s)S') & = &  \left( I(X_m(s)) \oplus I(S')
\right)/D(L_{(\pi_1)^m(s)}),\\
I(Y_m(s)S') & = &  \left( \left(I(Y_m(S))^{\vee} \oplus
I(S')^{\vee}\right)/D(L_{(\pi_2)^m(s)}) \right)^{\vee},
\end{eqnarray*}
respectively.
\medskip

{\it Bands.} Following the description of all bands in \S
\ref{linalg} and using polymerization we can
easily present  a band representation $I(P, \lambda, r)$ as a polymerization of a string representation.
Namely, for any sink $v_i$ of $P$ and an isomorphism $\sigma_i : {\mathbb C v_i} \to  {\mathbb C \bar{v}_i} $, we have $I(P, \lambda, r) \cong I(S(P,
v_i))^{(r)}(\lambda, \sigma_i)$.

\subsection{Socle series of the indecomposables} Recall that the socle
filtration of a module $M$ is the increasing filtration
$0 = \mbox{soc}_0 M\subset \mbox{soc}_1 M \subset ...
\subset \mbox{soc}_s M =  M$ uniquely defined by the property that
$\mbox{soc}_i M / \mbox{soc}_{i-1}M$ is the maximal semisimple
submodule of $M/\mbox{soc}_{i-1}M$.
Here we list all
semisimple quotients $\mbox{soc}_i M / \mbox{soc}_{i-1}M$ in the
socle series $0 = \mbox{soc}_0 M\subset \mbox{soc}_1 M \subset ...
\subset \mbox{soc}_s M =  M$ of any indecomposable $M$ in
${\mathcal C}^{0}_{\overline{\gamma(\mu})}$ ($s$ is the Loewy
length of $M$).

For a graded string or a graded directed polygon  $Q$ denote by $\mbox{Sink}(Q)$ the set of
sinks of $Q$. Set $Q^{(1)}:=Q$ and for $i>1$ let $Q^{(i)}$ be the
quiver obtained from $Q^{(i-1)}$ by removing all sinks and all
arrows whose tails are sinks of $Q^{(i-1)}$.

\begin{proposition}
(i) Let $S$ be a graded string. Then for $i\geq 1$,
$$\mbox{\rm soc}_i I(S) / \mbox{\rm  soc}_{i-1}I(S) \cong \bigoplus_{v \in {\rm Sink}(S^{(i)})}
L_{l(v)}.$$

(ii) Let $P$ be a graded directed polygon. Then  for $i\geq 1$,
$$ \mbox{\rm soc}_i I(P, \lambda, r) / \mbox{\rm soc}_{i-1}I(P,
\lambda, r)  \cong  \bigoplus_{v \in {\rm Sink}(P^{(i)})} L_{l(v)}^{\oplus r}.$$
\end{proposition}

\begin{example} The components of the socle series of the string module $I(S)$
corresponding to the string $S$ described in Example
\ref{exstring} are as follows:
\begin{eqnarray*}
\mbox{soc}_1 I(S) \cong L_1 \oplus L_2;& &\mbox{soc}_2 I(S)/
\mbox{soc}_1 I(S)  \cong L_1 \oplus L_2;\\
 \mbox{soc}_3 I(S)/
\mbox{soc}_2 I(S) \cong L_2;& & \mbox{soc}_4 I(S)/ \mbox{soc}_3
I(S) \cong L_1.
\end{eqnarray*}
\end{example}

\bibliography{ref,outref,mathsci}
\def\cprime{$'$} \def\cprime{$'$} \def\cprime{$'$} \def\cprime{$'$}
  \def\cprime{$'$}
\providecommand{\bysame}{\leavevmode\hbox
to3em{\hrulefill}\thinspace}
\providecommand{\MR}{\relax\ifhmode\unskip\space\fi MR }
\providecommand{\MRhref}[2]{%
  \href{http://www.ams.org/mathscinet-getitem?mr=#1}{#2}
} \providecommand{\href}[2]{#2}

\end{document}